\documentclass[11pt]{article}
\usepackage{amssymb,latexsym,amsmath,amsbsy,amsthm,amsxtra,amsgen,graphicx,
dsfont,xcolor,hyperref,enumitem,mathrsfs,hyperref,mathtools,stmaryrd}
\newtheorem{thm}{Theorem}[section]
\newtheorem{lemma}[thm]{Lemma}

\newtheorem{prop}[thm]{Proposition}

\theoremstyle{definition}

\newtheorem{rem}[thm]{Remark}

\newcommand{\N}{\mathbb{N}}

\newcommand{\Z}{\mathbb{Z}}
\newcommand{\E}{E}

\renewcommand{\P}{P}
\newcommand {\1} {\mathds{1}}

\newcommand{\Db}  {{\mathbb D}}

\newcommand{\Nb}  {{\mathbb N}}
\newcommand{\Rb}  {{\mathbb R}}

\newcommand{\cP}{\mathcal{P}}

\newcommand{\eps}{\varepsilon}
\newcommand{\om}{\omega}

\newcommand{\dd}{{\mathrm{d}}}

\numberwithin{equation}{section}

\begin{document}
\title{$\Lambda$-coalescents arising in populations with dormancy}
\author{Fernando Cordero\thanks{Technische Fakult\"at, Universit\"at Bielefeld.  Email: fcordero@techfak.uni-bielefeld.de}, Adri\'an Gonz\'alez Casanova\thanks{Instituto de Matem\'aticas, Universidad Nacional Aut\'onoma de M\'exico.  Email: adriangcs@matem.unam.mx}, \\ Jason Schweinsberg\thanks{Department of Mathematics, University of California San Diego. Email: jschweinsberg@ucsd.edu}, and Maite Wilke-Berenguer\thanks{Fakult\"at f\"ur Mathematik, Ruhr-Universit\"at Bochum. Email: maite.wilkeberenguer@ruhr-uni-bochum.de}}
\maketitle

\vspace{-.3in}

\begin{abstract}
Consider a population evolving from year to year through three seasons: spring, summer and winter.  Every spring starts with $N$ dormant individuals waking up independently of each other according to a given distribution.  Once an individual is awake, it starts reproducing at a constant rate.  By the end of spring, all individuals are awake and continue reproducing independently as Yule processes during the whole summer.  In the winter, $N$ individuals chosen uniformly at random go to sleep until the next spring, and the other individuals die.  We show that because an individual that wakes up unusually early can have a large number of surviving descendants, for some choices of model parameters the genealogy of the population will be described by a $\Lambda$-coalescent.  In particular, the beta coalescent can describe the genealogy when the rate at which individuals wake up increases exponentially over time.  We also characterize the set of all $\Lambda$-coalescents that can arise in this framework.
\end{abstract}

{\small {{\it AMS 2020 subject classifications}:  Primary 60J90;
Secondary 60J80, 92D15, 92D25}

{{\it Key words and phrases}:  seed bank, dormancy, $\Lambda$-coalescent}}

\allowdisplaybreaks

\section{Introduction}

Dormancy is a widespread evolutionary strategy. A significant proportion of all living microorganisms are in some form of latent state of life \cite{LJ}. The persistence of a latent subset of a population creates a buffer against selection \cite{GCetall, tell} and 
contributes to the maintenance of biodiversity \cite{LJ2}. Recently, probabilistic models have been successful in explaining some aspects of this facet of evolution and opening interesting questions in the fields of population genetics and beyond.  The impact of dormancy at different scales on the coalescent processes describing the genealogies of populations were investigated in \cite{ kkl,bgks,BGKW1, BGKW2}. Branching processes in random environment explain how dormancy can be selectively advantageous under fluctuating environmental conditions \cite{bhs20} while models from adaptive dynamics uncover that dormancy can arise from competition \cite{bt}.

While some microbial organisms can last dormant for millions of years \cite{mil}, most deactivation periods are much shorter. For example, mosquitos survive hostile environmental conditions by producing eggs which resist low temperatures and dry conditions \cite{mosco}. Water triggers the eclosion of the eggs, so typically the newborns have favorable weather conditions and avoid the dry season. The earlier an individual reaches the reproductive state the higher its chances of having a large number of descendants in the next generation. This induces a selective pressure favoring mosquitos that are fortunate to be in contact with water soon after the rainy season starts.

The mechanisms for leaving a dormant state have been observed to be under selective pressure not only in the case of mosquitos \cite{mosco2} but also in different contexts such as experimental evolution; see, for example, figure 2 of \cite{Lenski1} and \cite{Lag}. In such experiments, individuals reproduce until the resources are depleted, and after that some of them are sampled and propagated to fresh identical media. The process of taking some bacteria from a depleted to a fully replenished environment induces a form of latency. Similarly to the mosquitos, the bacteria undergoing the experiment are subjected to selective pressures resulting from the randomness of the activation time, and waking up rapidly from the dormant state provides an important advantage.

Wright and Vetsigian \cite{wv19} postulated that the randomness in the times when individuals emerge from a dormant state could cause the distribution of the numbers of offspring produced by different individuals to become highly skewed. Indeed, they demonstrated in their bacterial experiments that ``the heavy-tailed nature of the distribution of descendants can, in our case, be largely explained by phenotypic variability in lag time before exponential growth."  It is well established in the probability literature \cite{sag99, sch03} that heavy-tailed offspring distributions can affect the genealogical structure of the population.  Whereas the genealogy of populations under a wide range of conditions can be modeled by Kingman's coalescent \cite{king82}, in which only two lineages can merge at a time, the genealogies of populations with heavy-tailed offspring distributions are sometimes best described by coalescents with multiple mergers, also known as $\Lambda$-coalescents, which allow many ancestral lines to merge at once.  The primary aim of this paper is to describe how the randomness in the times when individuals emerge from dormancy affects the genealogy of the population and, in particular, to understand the conditions under which the genealogy is best described by a $\Lambda$-coalesent.  Wright and Vetsigian \cite{wv19} report that ``it is unlikely that the variance diverges with population size for the particular species and conditions we examined," indicating that $\Lambda$-coalescents probably do not arise in this instance. Nevertheless, we find that $\Lambda$-coalescent genealogies can appear if some rare individuals emerge from dormancy sufficiently early.

\subsection{A model involving dormancy}\label{modelsec}

We now describe a population model involving dormancy, which is very similar to the one introduced by Wright and Vetsigian in Section 2.7 of \cite{wv19}.  Note that we refer to the time period in the model as a day, which might be most appropriate when considering an evolutionary experiment involving bacteria, but in other contexts it will be more natural to think of this time period as lasting a year, and we will refer to different seasons.  The model evolves as follows.  We begin every day (or year) with a population of $N$ dormant individuals.  Each day (or year) has length $T_N$ and consists of three phases:
\begin{itemize}
\item {\bf Activation phase (Spring)}:  This phase has length $t_N$.  Each individual wakes up at some random time before $t_N$ and then starts reproducing at rate $\lambda_N$.

\item {\bf Active phase (Summer)}:  This phase has length $T_N - t_N$, and during this phase all individuals are awake and reproducing at rate $\lambda_N$.

\item {\bf Sampling phase (Winter)}:  At time $T_N$, we choose $N$ individuals uniformly at random from the population to go to sleep until the start of the next day (or year), and all other individuals die.
\end{itemize}
Note that we have compressed the winter into a single time point, whereas it may be more realistic to think of the deaths as occurring gradually over a longer winter period.  However, this will not substantially affect the results that follow.

For $d\in\Z$ and $i\in[N]\coloneqq \{1,\ldots,N\}$, we denote by $\tau_{i,N}^{(d)}$ the random time in $[0,t_N]$ when the $i$th individual at the start of day $d$ becomes active.  We assume that the random variables $\tau_{1,N}^{(d)}, \ldots, \tau_{N,N}^{(d)}$ are independent and identically distributed.  We denote by $X_{i,N}^{(d)}$ and $\nu_{i,N}^{(d)}$ the number of descendants of the $i$th individual starting day $d$ at time $T_N$ before and after sampling respectively. The total number of individuals at the end of day $d$ before sampling is then $$S_N^{(d)}\coloneqq \sum_{i\in[N]}X_{i,N}^{(d)}.$$  Note also that $\sum_{i\in[N]}\nu_{i,N}^{(d)}=N$.  Since the distributions of the $\tau_{i,N}^{(d)}$, $X_{i,N}^{(d)}$, and $\nu_{i,N}^{(d)}$ do not depend on $d$, we will frequently drop the superscripts when we are concerned only with the distributions of these quantities.  By a well-known fact about Yule processes, conditional on the activation time $\tau_{i,N}$, the random variable $X_{i,N}$ is geometrically distributed with parameter $\exp(-\lambda_N (T_N-\tau_{i,N}))$. Moreover, conditional on $(X_{1,N},\ldots, X_{N,N})$, we see that $\nu_{i,N}$ has a hypergeometric distribution with parameters $S_N$, $X_{i,N}$ and $N$.

We now explain how we will represent the genealogy of this population. Let us assume that we sample $n\in[N]$ individuals at random on day $0$. We define a discrete-time Markov chain $(\Psi_{n,N}(d))_{d=0}^\infty$ taking its values on the set of partitions $\cP_n$ of $[n]$, by letting $\Psi_{n,N}(d)$ be the partition of $[n]$ such that $i$ and $j$ are in the same block if and only if the $i$th and $j$th individuals in the sample have the same ancestor on day $-d$.  We will be interested in the asymptotic behaviour as $N\to\infty$ of this ancestral process.  The quantity
$$c_N\coloneqq \frac{ E[(\nu_{1,N})_2]}{N-1},$$
where $(\cdot)_n$ denotes the falling factorial, will play a crucial role. Note that $c_N$ is the probability that two individuals chosen uniformly at random from one generation have the same ancestor in the previous generation.  It therefore establishes the appropriate time scale on which to study the process because after scaling time by $1/c_N$, the expected time for two randomly chosen individuals to trace their lineages back to a common ancestor will equal 1.

Tools for studying the limits as $N \to \infty$ in models such as this one were developed in \cite{sag99, mosa01, sch03}.  It is well-known that when the distribution of the family sizes $\nu_{1,N}$ is highly skewed, the genealogy can sometimes be described by a $\Lambda$-coalescent. Recall that whenever $\Lambda$ is a finite measure on $[0,1]$, the $\Lambda$-coalescent, introduced independently in \cite{pit99} and \cite{sag99}, is a $\cP_n$-valued Markov process having the property that whenever there are $b$ blocks, each possible transition that involves $k$ of the blocks merging into one happens at rate $$\lambda_{b,k} \coloneqq \int_0^1 y^{k-2} (1-y)^{b-k} \: \Lambda({\rm d}y).$$  When $\Lambda$ is a unit mass at 0, we have $\lambda_{b,2} = 1$ and $\lambda_{b,k} = 0$ whenever $k \geq 3$, so we obtain Kingman's coalescent in which each pair of blocks merges at rate 1.  When $\Lambda(\{0\}) = 0$, the $\Lambda$-coalescent can be constructed from a Poisson point process on $[0, \infty) \times (0, \infty)$ with intensity ${\rm d}t \times y^{-2} \: \Lambda({\rm d}y)$, in such a way that if $(t, y)$ is a point of this Poisson point process, then at time $t$ we have a merger event in which each lineage independently participates in the merger with probability $y$.  This viewpoint is useful for understanding 
the results that follow.  In particular, when $\Lambda$ is a unit mass at one, we obtain a process known as the star-shaped coalescent in which all blocks merge after a waiting time whose distribution is exponential(1).

In Section \ref{ssec:simple}, we consider a simple case in which there is no summer, and individuals can only wake up at the very beginning and at the very end of the spring.  The genealogy of the population in this case is described in Theorem \ref{simplethm}.  In Section \ref{expsub}, we consider the case in which there is also no summer, but the rate at which individuals wake up from dormancy increases exponentially over time.  The genealogy of the population in this case is stated in Theorem \ref{expthm}.  In Section \ref{summersub}, we state Proposition~\ref{springsummer}, which gives conditions under which the summer period does not affect the genealogy of the population.  In Section \ref{charsub}, we state Theorem \ref{charthm}, which characterizes the possible $\Lambda$-coalescents that can arise as limits in this model.  After establishing some preliminary results in Section \ref{sec:cannings}, we prove Theorem~\ref{simplethm} in Section~\ref{2ptsec}, Theorem \ref{expthm} in Section~\ref{expsec}, Proposition~\ref{springsummer} in Section~\ref{summersec}, and Theorem \ref{charthm} in Section~\ref{classsec}.

\subsection{A two-point distribution for the exit time from dormancy}\label{ssec:simple}

Here we describe the limiting genealogy in a simple instance of the model introduced above in which there is no summer, meaning that $T_N = t_N$, and the random variables $\tau_{i,N}$ can take only the two values $0$ and $T_N$.  We write
\begin{equation*}
 \P(\tau_{i,N}=0)=\omega_N=1-\P(\tau_{i,N}=T_N),
\end{equation*}
where $\omega_N$ is assumed to satisfy 
\begin{equation}\label{cond0}
\lim_{N\to\infty} N\omega_N=0.
\end{equation}
Note that the progeny of individual $i$ at time $T_N$ is given by
$$X_{i,N}=b_{i,N} G_{i,N}+(1-b_{i,N}),$$
where $(b_{i,N})_{i=1}^N$ is an i.i.d. family of Bernoulli random variables with parameter $\omega_N$ and $(G_{i,N})_{i=1}^N$ is an i.i.d. family of geometric random variables with parameters $e^{-\lambda_N T_N}$ independent of $(b_{i,N})_{i=1}^N$. Let us assume in addition that
\begin{equation}\label{cond}
 \lambda_N T_N=\beta\log(\kappa N), \qquad \textup{for some }\kappa,\beta>0.
\end{equation}

We denote by $\Db_S[0, \infty)$ the set of c\`adl\`ag functions from $[0, \infty)$ to $S$, equipped with the usual Skorohod $J_1$ topology.  We then have the following result.  Throughout the paper, convergence of ancestral processes as $N \rightarrow \infty$ refers to weak convergence of stochastic processes in $\Db_{\cP_n}[0, \infty)$.

\begin{thm}\label{simplethm}
Assume that conditions \eqref{cond0} and \eqref{cond} hold.  Then:
\begin{enumerate}
\item If $\beta>1$, the processes $(\Psi_{n,N}(\lfloor t/c_N\rfloor))_{t \geq 0}$ converge as $N\to\infty$ to the star-shaped coalescent.
\item If $\beta=1$, the processes $(\Psi_{n,N}(\lfloor t/c_N\rfloor))_{t \geq 0}$ converge as $N\to\infty$ to the $\Lambda$-coalescent, with $\Lambda$ being the probability measure characterized by $y^{-2}\Lambda_\kappa({\rm d}y) := \frac{1}{\E[Y_\kappa^2]}\P(Y_\kappa \in {\rm d}y)$, where $Y_\kappa$ is a $[0,1]$-valued random variable whose distribution is determined by
\begin{equation}\label{Ydef}
\P(Y_\kappa>x)\coloneqq e^{- \frac{x}{\kappa(1-x)}},\qquad x\in[0,1].
\end{equation}
\item If $\beta<1$, the processes $(\Psi_{n,N}(\lfloor t/c_N\rfloor))_{t \geq 0}$ converge as $N\to\infty$ to Kingman's coalescent.
\end{enumerate}
\end{thm}

This result can be understood as follows.  Suppose an individual wakes up from dormancy unusually early, at time $0$ rather than at time $T_N$.  Then this individual spawns a branching process in which individuals give birth at rate $\lambda_N$, meaning that the expected number of descendants alive at time $T_N$ will be $e^{\lambda_N T_N}$.  Suppose first that $\beta = 1$.  Then the number of descendants alive at time $T_N$ can be approximately expressed as $\kappa N W$, where $W$ has an exponential(1) distribution.  Because there will also be $N-1$ individuals who wake up at time $T_N$, the fraction of the population at time $T_N$ that is descended from the individual who woke up early will be approximately $\kappa W/(\kappa W + 1)$, which means approximately a fraction $\kappa W/(\kappa W + 1)$ of the ancestral lines will merge at this time.  Indeed it is straightforward to check using (\ref{Ydef}) that we have the equality in distribution
\begin{equation}\label{eq:Y_is_W}
 Y_{\kappa} =_d \frac{\kappa W}{\kappa W + 1}.
\end{equation}

If instead $\beta < 1$, then the number of descendants at time $T_N$ of the individual who woke up early will be much smaller than $N$, making multiple mergers unlikely and giving rise to a Kingman's coalescent genealogy.  On the other hand, if $\beta > 1$, then the number of descendants of this individual will be much larger than $N$, meaning that with high probability all of the sampled ancestral lines merge, leading to a star-shaped genealogy.

\subsection{Exponentially increasing rates of exit from dormancy}\label{expsub}

We now consider a possibly more realistic scenario in which the rate at which individuals exit from dormancy increases approximately exponentially over time.  As in the model in Section \ref{ssec:simple}, we assume there is no summer, so $T_N = t_N$.  We assume that $\lambda_N = \lambda > 0$ for all $N$.  We also assume that $T_N - \tau_{i,N} = \zeta_i \wedge T_N$, where $(\zeta_i)_{i=1}^N$ is a sequence of i.i.d. random variables whose distribution does not depend on $N$.  We assume there exist constants $\gamma > 0$ and $c > 0$ such that, using $\sim$ to denote that the ratio of the two sides tends to one, we have
\begin{equation}\label{zetatail}
\P(\zeta_1 > y) \sim  c e^{-\gamma y}, \quad\textrm{as $y\to\infty$,}
\end{equation}
Finally, we assume that
\begin{equation}\label{zetaTN}
  \lim_{N\to\infty}\frac{\log N}{T_N}=0.
\end{equation}
We obtain the following result.

\begin{thm}\label{expthm}
Assume that conditions (\ref{zetatail}) and (\ref{zetaTN}) hold, and let $a = \gamma/\lambda$. Then:
\begin{enumerate}
\item If $a \geq 2$, the processes $(\Psi_{n,N}(\lfloor t/c_N\rfloor))_{t \geq 0}$ converge as $N\to\infty$ to Kingman's coalescent.
\item If $1 \leq a < 2$, the processes $(\Psi_{n,N}(\lfloor t/c_N\rfloor))_{t \geq 0}$ converge as $N\to\infty$ to the $\Lambda$-coalescent in which $\Lambda$ is the Beta$(2-a, a)$ distribution.
\item If $0< a < 1$, the processes $(\Psi_{n,N}(\lfloor d \rfloor))_{d =0}^{\infty}$ converge as $N\to\infty$ to the discrete-time $\Xi_{a}$-coalescent described in Theorem 4(d) of \cite{sch03}.
 \end{enumerate}
\end{thm}

To understand this result, consider the distribution of the random variables $X_{i,N}$.  If we disregard the truncation at $T_N$, which will turn out to have minimal effect, then conditional on $\zeta_i = u$, the distribution of $X_{i,N}$ is geometric with parameter $e^{-\lambda u}$.   We assume for simplicity that the distribution of $\zeta_i$ is exactly exponential with rate $\gamma$. Then, making the change of variables $s = e^{-u \lambda}$ and using Stirling's approximation in the last step, we obtain
$$\P(X_{i,N}>k)=\int_0^{\infty}(1-e^{-u\lambda})^k\gamma e^{-\gamma u}{\rm d}u= a \int_{0}^{1}(1-s)^ks^{a-1}{\rm d}s= \frac{a \Gamma(k+1) \Gamma(a)}{\Gamma(k+1+a)} \sim \Gamma(a+1) k^{-a}.$$  Therefore, this result fits into the framework of
Theorem 4 of \cite{sch03}, where it was established that beta coalescents describe the limiting genealogies in populations with these heavy-tailed offspring distributions.  Note also that when $0 < a < 1$, there will be multiple individuals in each day whose descendants comprise a substantial fraction of the population, which is why the limit of the ancestral processes is not a $\Lambda$-coalescent but rather a discrete-time process in which multiple groups of lineages merge in each time step.

\subsection{The effect of the summer on the genealogy}\label{summersub}
In Sections \ref{ssec:simple} and \ref{expsub}, we assumed there was no summer. We now consider how the inclusion of a summer period impacts the genealogy of the population.  For this, we compare two populations whose spring has length $t_N$ and whose reproduction rate is $\lambda_N$, which are subject to the same activation times.  One population has a summer of length $T_N - t_N>0$, and the other one has no summer, so that years have length $t_N$.
Let $(\hat{\Psi}_{n,N}(d))_{d=0}^\infty$ denote the ancestral process associated with a sample of $n$ individuals at time $0$ from the population without a summer, and let $(\Psi_{n,N}(d))_{d=0}^\infty$ denote the corresponding ancestral process for the model with a summer.  We obtain the following result.

\begin{prop}\label{springsummer}
Assume that there is a sequence $(\rho_N)_{N=1}^{\infty}$ of positive numbers such that $(\hat{\Psi}_{n,N}(\lfloor \rho_N t \rfloor))_{t\geq 0}\Rightarrow (\Psi_n(t))_{t\geq 0}$ in $\Db_{\cP_n}[0, \infty)$, where $\Psi_n$ is a continuous-time Markov chain with values in $\cP_n$.  For the model with a summer, let $Y_{i,N}$ denote the number of descendants of the $i$th individual at the beginning of a given season who are alive at time $t_N$, and assume that
\begin{equation}\label{H}
\lim_{N \rightarrow \infty} \rho_N E\left[\frac{1}{ Y_{1,N}+\cdots+Y_{N,N}}\right] = 0.
\end{equation}
Then 
$$
({\Psi}_{n,N}(\lfloor \rho_N t\rfloor))_{t\geq 0}\Rightarrow (\Psi_n(t))_{t\geq 0}.
$$
\end{prop}

When condition (\ref{H}) holds, coalescence of lineages during the summer is sufficiently rare that the inclusion of the summer period does not affect the genealogy of the population in the limit.  To understand why (\ref{H}) is the correct condition, note that under the usual scaling $\rho_N = 1/c_N$, the probability that two randomly chosen lineages coalesce is $c_N = 1/\rho_N$.  On the other hand, conditional on there being $M$ individuals alive at the end of the spring, the probability that two randomly chosen lineages merge during the summer, regardless of the length of the summer, is bounded above by $2/(M + 1)$, as a consequence of Lemma \ref{polya}. Therefore, condition (\ref{H}) implies that the probability that two lineages coalesce in the summer is of smaller order than the probability that two lineages coalesce in the spring. 

On the other hand, if (\ref{H}) fails, then as long as the summer has a length that does not tend to zero as $N \rightarrow \infty$, the summer period will cause additional pairwise mergers of ancestral lines.

\subsection{\texorpdfstring{A characterization of the possible $\Lambda$-coalescent limits}{}}\label{charsub}

Theorem \ref{charthm} below characterizes the $\Lambda$-coalescents that can arise as limits of ancestral processes in the model introduced in Section \ref{modelsec}.  Essentially, the possible measures $\Lambda$ are those that can be expressed as mixtures of the measures $\Lambda$ that appear as limits in Theorem \ref{simplethm}.  Note that the model in Section \ref{modelsec} is characterized by the birth rates $(\lambda_N)_{N=1}^{\infty}$, the times $(t_N)_{N=1}^{\infty}$ and $(T_N)_{N=1}^{\infty}$, and the distributions of the random variables $(\tau_{1,N})_{N=1}^{\infty}$.  Therefore, to say that the $\Lambda$-coalescent can arise as a limit in this model means that there are choices of these parameters for which the rescaled ancestral processes converge to the $\Lambda$-coalescent.

\begin{thm}\label{charthm}
It is possible for the $\Lambda$-coalescent to arise as the limit of the rescaled ancestral processes $(\Psi_{n,N}(\lfloor \rho_N t \rfloor))_{t \geq 0}$ in the population model defined in Section \ref{modelsec} if and only if we can write
\begin{equation}\label{decomp}
 \Lambda = a_1 \delta_0 + a_2 \delta_1 + \Lambda',
\end{equation}
where $a_1$ and $a_2$ are nonnegative real numbers and $\Lambda'$ is a measure on $(0,1)$ with density $h$ with respect to Lebesgue measure, where
\begin{equation}\label{hdef}
h(y) = \int_0^{\infty} \frac{1}{\kappa} \bigg(\frac{y}{1-y} \bigg)^2 e^{- \frac{y}{\kappa(1-y)}} \: \eta({\rm d} \kappa)
\end{equation}
for all $y \in (0,1)$ and $\int_0^{\infty} (1 \wedge \kappa^2) \: \eta({\rm d} \kappa) < \infty$.
\end{thm}

We now make a few remarks concerning this result:
\begin{enumerate}
\item To relate this result to Theorem \ref{simplethm}, note that the density of the random variable $Y_\kappa$ described in (\ref{Ydef}) is given by $$f_{\kappa}(y) = \frac{1}{\kappa (1-y)^2} e^{- \frac{y}{\kappa(1-y)}},\qquad y\in[0,1].$$  Therefore, if $\Lambda'$ has density $h$ with respect to Lebesgue measure, where $h$ is given by (\ref{hdef}), then $y^{-2} \Lambda'({\rm d}y)/{\rm d}y = \int_0^{\infty} f_{\kappa}(y) \: \eta({\rm d}\kappa)$. Note also that
\begin{equation}\label{declp}
 \Lambda'({\rm d}y)=h(y){\rm d}y=\left(\int_0^\infty y^2 f_\kappa(y)\eta({\rm d} \kappa)\right){\rm d}y=\int_0^\infty\Lambda_\kappa({\rm d}y)E[Y_\kappa^2]\eta({\rm d}\kappa).
\end{equation}

\item Suppose $\eta({\rm d} \kappa) = \kappa^{-1-a} \: {\rm d} \kappa$, where $0 < a < 2$.  Then $\int_0^{\infty} (1 \wedge \kappa^2) \: \eta({\rm d} \kappa) < \infty$.  In this case, making the substitution $x = y/(\kappa(1-y))$, we have
$$h(y) = \bigg(\frac{y}{1-y} \bigg)^2 \int_0^{\infty} e^{- \frac{y}{\kappa(1-y)}} \kappa^{-2-a} \: {\rm d} \kappa = \bigg(\frac{y}{1 - y} \bigg)^{1 - a} \int_0^{\infty} e^{-x} x^a \: {\rm d}x,$$ which is a constant multiple of the Beta$(2-a, a)$ density appearing in Theorem \ref{expthm}.

\item Typically, when rescaling time in the ancestral processes, we take $\rho_N = 1/c_N$, so that the expected time required for two randomly chosen lineages to merge equals $1$.  When $\rho_N = 1/c_N$, the measure $\Lambda$ that appears in the limit must be a probability measure.  Allowing arbitrary scaling constants $(\rho_N)_{N=1}^{\infty}$ simply allows finite measures $\Lambda$ that are not probability measure to arise in the limit.  See Remark \ref{probrem} for details.

\item For the purposes of this result, we may assume without loss of generality that $t_N = T_N$, because regardless of the choice of $t_N$, it is always possible that the distribution of the times $\tau_{i,N}$ may be concentrated on a smaller subinterval of $[0, T_N]$.  That is, every $\Lambda$-coalescent that can arise in this model can arise in an instance of the model in which there is no summer period.

\item The decomposition \eqref{hdef} in Theorem \ref{charthm} is unique. Indeed, if two measures $\eta$ and $\hat{\eta}$ define the same function $h$ via \eqref{hdef}, then the Laplace transforms of the measures $z \eta \circ \phi^{-1}({\rm d}z)$ and $z \hat{\eta}\circ \phi^{-1}({\rm d}z)$, where $\phi(\kappa)=1/\kappa$, coincide in $(0,\infty)$. We conclude that $\eta\circ \phi^{-1}=\hat{\eta}\circ \phi^{-1}$, and hence $\eta=\hat{\eta}$. In combination with Theorem \ref{charthm} we infer that $\{\delta_0,\delta_1\}\cup \{\Lambda_\kappa\}_{\kappa>0}$ is the set of extremal points of the convex set of probability measures $\Lambda$ appearing in Theorem~\ref{charthm}.  Note that (\ref{declp}) shows how any probability measure $\Lambda$ appearing in Theorem~\ref{charthm} can be written as a mixture of the probability measures $\Lambda_{\kappa}$.
\end{enumerate}

We end this section with an alternate characterization of the measures $\Lambda'$ appearing in Theorem~\ref{charthm}.
\begin{prop}
 An integrable function $h:(0,1)\to(0,\infty)$ can be expressed as in \eqref{hdef} for some measure $\eta$ on $[0,\infty)$ if and only if there is a completely monotone function $g:(0,\infty)\to(0,\infty)$ such that 
 $$h(y)=\frac{y^2}{(1-y)^2}\,g\left(\frac{y}{1-y}\right),$$
 for all $y\in(0,1)$ and $\int_0^\infty g(v) (1\wedge v^2)\: {\rm d}v<\infty$.
\end{prop}
\begin{proof}
First note that, the relation $Tg(y)=(y/(1-y))^2g(y/(1-y))$, $y\in(0,1)$, defines a bijective function $T$ mapping positive functions on $(0,\infty)$ into positive functions on $(0,1)$. Moreover, making the substitution $v=y/(1-y)$, we see that
 $$\int_0^1 \frac{y^2}{(1-y)^2}\,g\left(\frac{y}{1-y}\right)\:{\rm d}y=\int_0^\infty \frac{v^2}{(1+v)^2}g(v)\:{\rm d}v.$$
 Hence, the integrability of $Tg$ is equivalent to the integrability of $v\mapsto g(v) (1\wedge v^2)$.
 In addition, if $h$ is given by \eqref{hdef}, then $h=Tg$ with  
 $$g(v)=\int_0^\infty \frac1{\kappa} e^{-\frac{v}{\kappa}}\: \eta({\rm d} \kappa).$$
Clearly $g$ is the Laplace transform of the measure $m({\rm d}z)\coloneqq z\: \eta \circ \phi^{-1}({\rm d}z)$, where $\phi(\kappa)=1/\kappa$, and thus, $g$ is completely monotone. Conversely, assume that $h=Tg$ for some completely monotone function $g$. The latter is, by Bernstein's theorem, the Laplace transform of a measure $m$ on $[0,\infty)$. It follows that $h$ is given by \eqref{hdef} with $\eta({\rm d}\kappa)\coloneqq \kappa\: m \circ \phi^{-1}({\rm d}\kappa)$, which ends the proof.
 \end{proof}

\section{Genealogies in Cannings Models}\label{sec:cannings}

The model introduced in Section \ref{modelsec} is an example of a Cannings model.  Cannings models, first introduced in \cite{can1, can2}, have discrete generations and a fixed population size $N$, and the distribution of the family size vector $(\nu_{1,N}, \dots, \nu_{N,N})$ is required to be exchangeable.  There is by now a standard set of tools for studying the genealogy of such models.  To prove convergence of the ancestral process in our model to a $\Lambda$-coalescent, we will mainly use the following result, which is essentially Theorem 3.1 of \cite{sag99}; see also Theorem 3.1 of \cite{sch03}.

\begin{thm}\label{Cannings}
Consider a Cannings model in which the sizes of the $N$ families are denoted by $\nu_{1,N}, \dots, \nu_{N,N}$.  Let $$c_N = \frac{E[(\nu_{1,N})_2]}{N-1}.$$ Now sample $n$ individuals at random from the population in generation zero, and let $\Psi_{n,N}(k)$ be the partition of $[n]$ such that $i$ and $j$ are in the same block of $\Psi_{n,N}(k)$ if and only if the $i$th and $j$th individuals in the sample have the same ancestor in generation $-k$.  Let $\Lambda$ be a probability measure on $[0,1]$.  Then the processes $(\Psi_{n,N}(\lfloor t/c_N\rfloor))_{t \geq 0}$ converge in $\Db_{{\cP}_n}[0, \infty)$ to the $\Lambda$-coalescent for all $n\in\Nb$ if and only if the following three conditions hold:
\begin{enumerate}
\item We have $$\lim_{N \rightarrow \infty} c_N = 0.$$

\item We have $$\lim_{N \rightarrow \infty} \frac{E[(\nu_{1,N})_2(\nu_{2,N})_2]}{N^2 c_N} = 0.$$

\item For all $x \in (0, 1)$ such that $\Lambda(\{x\}) = 0$, we have
$$\lim_{N \rightarrow \infty} \frac{N}{c_N}\P(\nu_{1,N} > Nx) = \int_x^1 y^{-2} \: \Lambda({\rm d}y).$$
\end{enumerate}
\end{thm}

\begin{rem} In place of condition 1 in Theorem \ref{Cannings}, Sagitov \cite{sag99} has the condition that $N^{-1} E[(\nu_{1,N} - 1)^2] \rightarrow 0$, which is equivalent because the fact that $E[\nu_{1,N}] = 1$ implies that $E[(\nu_{1,N})_2] = E[(\nu_{1,N} - 1)^2]$.  Also, in place of condition 3, Sagitov has the condition that for all integers $a \geq 2$, we have $$\lim_{N \rightarrow \infty} \frac{E[(\nu_{1,N} - 1)^2 \cdots (\nu_{a,N} - 1)^2]}{N^a c_N} = 0.$$  However, the equivalence of (16) and (20) in \cite{mosa01} implies that we can consider $(\nu_{k,N})_2$ in place of $(\nu_{k,N} - 1)^2$, and equation (17) of \cite{mosa01} implies that the limit is zero for all $a \geq 2$ if and only if the limit is zero when $a = 2$.  These observations lead to the formulation of the result given above.
\end{rem}
For the rest of this section, we consider a subclass of Cannings models in which the family sizes in each generation are obtained in the following way.  We consider a sequence of independent and identically distributed positive integer-valued random variables  $X_{1,N}, \dots, X_{N,N}$, where $X_{k,N}$ denotes the number of offspring produced by the $k$th individual.  Note that we do not allow the random variables $X_{k,N}$ to take the value zero.  We let $S_N = X_{1,N} + \dots + X_{N,N}$ be the total number of offspring.  We then sample $N$ of the $S_N$ offspring without replacement to form the next generation, and denote by $\nu_{k,N}$ the number of offspring of the $k$th individual that are sampled.  Note that the model introduced in Section \ref{modelsec} fits into this framework.

In this section, we will use the notation $f(N) \ll g(N)$ to mean $\lim_{N \rightarrow \infty} f(N)/g(N) > 0$ and $f(N) \lesssim g(N)$ to mean $\sup_N f(N)/g(N) < \infty$.

Lemma \ref{dict} is useful for translating properties of the family sizes after sampling to properties of the family sizes before sampling, and vice versa.

\begin{lemma} \label{dict}
For $r\in\Nb$, $k_1,\ldots, k_r\geq 2$, and $N>k_1+\cdots+k_r$, we have
\begin{equation}\label{multimoments}
\frac{ E[(\nu_{1,N})_{k_1}\cdots({\nu}_{r,N})_{k_r}]}{(N)_{k_1+\cdots+k_r}}= E\left[\frac{(X_{1,N})_{k_1}\cdots(X_{r,N})_{k_r}}{({S}_N)_{k_1+\cdots+k_r}}\right].
\end{equation}
In particular,
 \begin{equation}\label{cNX}
c_N \coloneqq \frac{E[(\nu_{1,N})_2]}{N-1} = N E \bigg[ \frac{X_{1,N}(X_{1,N} - 1)}{S_N(S_N - 1)} \bigg].
\end{equation}
Moreover,
\begin{equation}\label{lowcN}
c_N\geq \frac{\P(X_{1,N}\geq 2)^2}{2N}.
\end{equation}
\end{lemma}

\begin{proof}
The identity (\ref{multimoments}) follows from equations (19) and (21) in the proof of Lemma 6 of \cite{sch03}.
Both sides give the probability that, if we sample $k_1 + \dots + k_r$ individuals, the first $k_1$ are descended from the first individual in the previous generation, the next $k_2$ are descended from the second individual in the previous generation, and so on.  The result (\ref{cNX}) is a special case of (\ref{multimoments}).
It remains to prove (\ref{lowcN}). Using \eqref{cNX} for the first inequality and Jensen's inequality for the third, we have
\begin{align*}
 c_N &\geq N E\left[\frac{X_{1,N}(X_{1,N}-1)}{S_N^2}\right]\geq \frac{N}{2} E\left[\frac{X_{1,N}^2}{S_N^2}\1_{\{X_{1,N}\geq 2\}}\right] \\
 &\geq \frac{N}{2}\left(E\left[\frac{X_{1,N}}{S_N}\1_{\{X_{1,N}\geq 2\}}\right]\right)^2 \geq \frac{N}{2}\left(\frac{1}{N}\P(X_{1,N}\geq 2)\right)^2,
\end{align*}
which ends the proof.
 \end{proof}

Lemma \ref{iidXS} shows that when $c_N \rightarrow 0$, the distribution of the total number of offspring is highly concentrated around some value $a_N$ when $N$ is large.  This result implies that, when an unusually large family arises that will produce multiple mergers of ancestral lines, the total size of the remaining $N-1$ families in that generation can be treated as being essentially nonrandom, so that the size of the large family will determine the proportion of ancestral lineages that merge in this generation.

\begin{lemma}\label{iidXS}
Suppose $\lim_{N \rightarrow \infty} c_N = 0$.  Then there exists a sequence of positive numbers $(a_N)_{N=1}^{\infty}$ such that the following hold:
\begin{enumerate}
\item We have $S_N/a_N \rightarrow_p 1,$ where $\rightarrow_p$ denotes convergence in probability as $N \rightarrow \infty$.

\item There exists $\delta > 0$ such that $$\lim_{N \rightarrow \infty} c_N^{-1} \P\Big(S_N - \max_{1 \leq k \leq N} X_{k,N} < \delta a_N\Big) = 0.$$
\end{enumerate}
\end{lemma}

\begin{proof}
Let $\bar{S}_{N-1} = S_N - X_{1,N} = X_{2,N} + \dots + X_{N,N}$.
Let
\begin{equation}\label{mndef}
m_N = \inf\{k:\P(\bar{S}_{N-1} \leq k) \geq 1/2\}.
\end{equation}
Let $0 < \delta < 1/2$.  We have
$$\P\big(X_{1,N} \geq \delta \bar{S}_{N-1} \big) \geq\P(X_{1,N} \geq \delta m_N)\P(\bar{S}_{N-1} \leq m_N)
\geq \frac{1}{2}\P(X_{1,N} \geq \delta m_N).$$
Because $\bar{S}_{N-1} = S_N - X_{1,N}$, it follows that
\begin{equation}\label{NPX}
N\P(X_{1,N} \geq \delta m_N) \leq 2N \P\big(X_{1,N} \geq \delta \bar{S}_{N-1}\big) = 2N \P \bigg( X_{1,N} \geq \frac{\delta}{1 + \delta} S_N \bigg).
\end{equation}
Using the fact that $S_N \geq N$ followed by Markov's Inequality, we have
\begin{align}\label{PXS}
\P \bigg( X_{1,N} \geq \frac{\delta}{1 - \delta} S_N \bigg) &= \P \bigg( X_{1,N}(X_{1,N} - 1) \geq \Big(\frac{\delta}{1 + \delta} S_N \Big) \Big( \frac{\delta}{1 + \delta} S_N - 1 \Big) \bigg) \nonumber \\
&= \P \bigg( X_{1,N}(X_{1,N} - 1) \geq \frac{\delta^2}{(1 + \delta)^2} S_N(S_N - 1) \bigg(\frac{S_N - \frac{1+\delta}{\delta}}{S_N - 1} \bigg) \bigg) \nonumber \\
&\leq \P \bigg( \frac{X_{1,N}(X_{1,N} - 1)}{S_N(S_N - 1)} \geq \frac{\delta^2}{(1 + \delta)^2} \cdot \frac{N - \frac{1+\delta}{\delta}}{N - 1} \bigg) \nonumber \\
&\leq \bigg(\frac{(1 + \delta)^2}{\delta^2} \cdot \frac{N - 1}{N - \frac{1+\delta}{\delta}} \bigg) E \bigg[\frac{X_{1,N}(X_{1,N} - 1)}{S_N(S_N - 1)} \bigg].
\end{align}
From (\ref{cNX}), (\ref{NPX}), and (\ref{PXS}) we get
\begin{equation}\label{Xtail}
N\P(X_{1,N} \geq \delta m_N) \leq \bigg(\frac{2 (1 + \delta)^2}{\delta^2} \cdot \frac{N - 1}{N - \frac{1+\delta}{\delta}} \bigg) c_N,
\end{equation}
which tends to zero as $N \rightarrow \infty$ by assumption.

For $1 \leq k \leq N$, let $$W_{k,N} = X_{k,N}\1_{\{X_{k,N} \leq \delta m_N\}} + \1_{\{X_{k,N} > \delta m_N\}}.$$ Let $U_N = W_{1,N} + \dots + W_{N,N}$, and let $\bar{U}_{N-1} = U_N - W_{1,N}$.  Equation (\ref{Xtail}) implies
\begin{equation}\label{STcompare}
\P(S_N \neq U_N) \leq N \P(X_{1,N} \neq W_{1,N}) \rightarrow 0 \quad \mbox{as }N \rightarrow \infty.
\end{equation}
Let $d_N = E[W_{1,N}]$, and let $a_N = N d_N$.  Then $E[U_N] = a_N$.
It follows from (\ref{mndef}) that $$\P(\bar{S}_{N-1} \leq m_N-1) < 1/2,$$ and therefore by Markov's Inequality,
\begin{align*}
1/2 &< \P(\bar{S}_{N-1} \geq m_N) \\
&\leq \P(\bar{U}_{N-1} \geq m_N) + (N-1)\P(X_{1,N} \neq W_{1,N}) \\
&\leq \frac{E[\bar{U}_{N-1}]}{m_N} + (N-1)\P(X_{1,N} \neq W_{1,N}) \\
&= \frac{(N-1)a_N}{Nm_N} + (N-1)\P(X_{1,N} \neq W_{1,N}).
\end{align*}
Because $(N-1) \P(X_{1,N} \neq W_{1,N}) \rightarrow 0$ by (\ref{STcompare}), it follows that for sufficiently large $N$, we have
\begin{equation}\label{amratio}
\frac{a_N}{m_N} \geq \frac{1}{3}.
\end{equation}

We now establish the result using a second moment argument.  Given $r = 1/\ell$ for some positive integer $\ell$ and an integer $k$ such that $1 \leq k \leq \ell$, define
$$U_{k,r,N} = W_{\lfloor (k-1)rN \rfloor + 1, N} + \dots + W_{\lfloor krN \rfloor, N}.$$
Note that $U_N = U_{1,1,N}$, and that $U_{k,r,N}$ is the sum of between $rN - 1$ and $rN + 1$ random variables.  Therefore,
\begin{equation}\label{ETk}
E[U_{k,rN}] \geq \Big(r - \frac{1}{N} \Big) a_N,
\end{equation}
and using that $W_{1,N} \geq 1$, we have
\begin{align}\label{VarT}
\mbox{Var}(U_{k,r,N}) &\leq (rN + 1)\mbox{Var}(W_{1,N}) \nonumber \\
&= (rN + 1) \big(E[W_{1,N}^2] - (E[W_{1,N}])^2\big) \nonumber \\
&\leq (rN + 1)E[W_{1,N}(W_{1,N} - 1)].
\end{align}
Using that $X_{1,N} \leq \delta m_N$ on the event $\{W_{1,N} - 1 \neq 0\}$, and that $W_{1,N}$ is independent of $\bar{S}_{N-1}$, we have
\begin{align}\label{Yfac1}
c_N &\geq N E \bigg[ \frac{W_{1,N}(W_{1,N} - 1)}{S_N^2} \bigg] \geq N E \bigg[ \frac{W_{1,N}(W_{1,N} - 1)}{(\delta m_N + \bar{S}_{N-1})^2} \bigg] \nonumber \\
&= N E[W_{1,N}(W_{1,N} - 1)] E \bigg[ \frac{1}{(\delta m_N + \bar{S}_{N-1})^2} \bigg].
\end{align}
Now
\begin{equation}\label{Yfac2}
E \bigg[ \frac{1}{(\delta m_N + \bar{S}_{N-1})^2} \bigg] \geq \P(\bar{S}_{N-1} = \bar{U}_{N-1}) E \bigg[ \frac{1}{(\delta m_N + \bar{S}_{N-1})^2} \Big| \bar{S}_{N-1} = \bar{U}_{N-1} \bigg].
\end{equation}
Let $Z_{1,N}, \dots, Z_{N,N}$ be independent random variables whose distribution is the conditional distribution of $X_{1,N}$ given $X_{1,N} \leq \delta m_N$.  Note that
\begin{equation}\label{Zmean}
E[Z_{1,N}] = E[W_{1,N}|X_{1,N} \leq \delta m_N] \leq \frac{E[W_{1,N}]}{\P(X_{1,N} \leq \delta m_N)} \leq \frac{a_N}{N \P(X_{1,N} \leq \delta m_N)}.
\end{equation}
Let $\bar{V}_{N-1} = Z_{2,N} + \dots + Z_{N,N}$.  Then, using Jensen's Inequality, followed by (\ref{Zmean}) and then (\ref{amratio}), for $N$ large enough that $\delta m_N > 1$ we have
\begin{align*}
E \bigg[ \frac{1}{(\delta m_N + \bar{S}_{N-1})^2} \Big| \bar{S}_{N-1} = \bar{U}_{N-1} \bigg] &= E \bigg[ \frac{1}{(\delta m_N + \bar{V}_{N-1})^2} \bigg] \\
&\geq \frac{1}{(\delta m_N + E[\bar{V}_{N-1}])^2} \\
&\geq \bigg(\delta m_N + \frac{(N-1) a_N}{N\P(X_{1,N} \leq \delta m_N)} \bigg)^{-2} \\
&\geq \frac{1}{a_N^2} \bigg(3 \delta + \frac{N-1}{N \P(X_{1,N} \leq \delta m_N)} \bigg)^{-2}.
\end{align*}
It follows from (\ref{STcompare}) that $\lim_{N \rightarrow \infty} (N-1)/(N\P(X_{1,N} \leq \delta m_N)) = 1$, and therefore for sufficiently large $N$, we have
$$E \bigg[ \frac{1}{(\delta m_N + \bar{S}_{N-1})^2} \Big| \bar{S}_{N-1} = \bar{U}_{N-1} \bigg] \geq \frac{1}{a_N^2(1 + 4 \delta)^2}.$$
Combining this result with (\ref{Yfac1}) and (\ref{Yfac2}), we get
$$c_N \geq N E[W_{1,N}(W_{1,N} - 1)] \P(\bar{S}_{N-1} = \bar{U}_{N-1}) \cdot \frac{1}{a_N^2 (1 + 4 \delta)^2},$$
and therefore
$$E[W_{1,N}(W_{1,N} - 1)] \leq \frac{c_N a_N^2 (1 + 4 \delta)^2}{N \P(\bar{S}_{N-1} = \bar{U}_{N-1})}.$$
By (\ref{VarT}),
\begin{equation}\label{VarTk}
\mbox{Var}(U_{k,r,N}) \leq \frac{(rN + 1)c_N a_N^2 (1 + 4 \delta)^2}{N \P(\bar{S}_{N-1} = \bar{U}_{N-1})}.
\end{equation}

Let $\eps > 0$.  Taking $r = 1$ and using Chebyshev's Inequality along with (\ref{STcompare}) and the assumption that $\lim_{N \rightarrow \infty} c_N = 0$, we get
$$\P \bigg( \bigg|\frac{S_N}{a_N} - 1 \bigg| > \eps \bigg) \leq \P(S_N \neq U_N) + \frac{\mbox{Var}(U_N)}{a_N^2 \eps^2} \leq \P(S_N \neq U_N) + \frac{(N+1) c_N (1 + 4 \delta)^2}{N\P(\bar{S}_{N-1} = \bar{U}_{N-1}) \eps^2} \rightarrow 0.$$
That is, we have $S_N/a_N \rightarrow_p 1$, which is part 1 of the result.  To prove part 2, we take $r = 1/3$ and note that we can have $$S_N - \max_{1 \leq k \leq N} X_{k,N} < \delta a_N$$ only if at least two of the random variables $U_{k,r,N}$, for $k \in \{1, 2, 3\}$, are less than $\delta a_N$.  Let $0 < \delta < 1/3$.  By (\ref{ETk}), (\ref{VarTk}), and Chebyshev's Inequality,
\begin{align*}
\P(U_{k,r,N} < \delta a_N) &\leq \P\bigg( \big|U_{k,r,N} - E[U_{k,r,N}] \big| > \Big(r - \delta - \frac{1}{N} \Big)a_N \bigg) \\
&\leq \Big(r - \delta - \frac{1}{N} \Big)^{-2} \frac{\mbox{Var}(U_{k,r,N})}{a_N^2} \\
&\leq \Big(r - \delta - \frac{1}{N} \Big)^{-2} \frac{(rN + 1)c_N (1 + 4 \delta)^2}{N \P(\bar{S}_{N-1} = \bar{U}_{N-1})}.
\end{align*}
Therefore, the probability that at least two of the random variables $U_{k,r,N}$, for $k \in \{1, 2, 3\}$, are less than $\delta a_N$ is bounded above by
$$3 \Big(r - \delta - \frac{1}{N} \Big)^{-4} \cdot \frac{(rN + 1)^2c_N^2 (1 + 4 \delta)^4}{N^2 \P(\bar{S}_{N-1} = \bar{U}_{N-1})^2}.$$
This expression tends to zero faster than $c_N$ because $\delta < 1/3$.
\end{proof}

\begin{lemma}\label{LLNlem}
Suppose $\lim_{N \rightarrow \infty} c_N = 0$.  For all $\eps > 0$ we have
$$\lim_{N \rightarrow \infty} \frac{N^2}{c_N} \P \bigg( \bigg| \frac{\nu_{k,N}}{N} - \frac{X_{k,N}}{S_N} \bigg| > \eps \mbox{ for some }k \in [N] \bigg) = 0.$$
\end{lemma}

\begin{proof}
Suppose $Y$ is the number of red balls drawn, when $n$ balls are chosen from an urn containing $b$ balls, of which $r$ are red.  Chv\'atal \cite{chv} showed that if $\eps > 0$, then
\begin{equation}\label{hypg}
\P \bigg(Y \geq \Big(\frac{r}{b} + \eps \Big)n \bigg) \leq e^{-2 \eps^2 n}.
\end{equation}
Conditional on $X_{k,N}$ and $S_N$, we see that $\nu_{k,N}$ can be interpreted as the number of red balls drawn, when $N$ balls are chosen from an urn containing $S_N$ balls, of which $X_{k,N}$ are red.  By applying (\ref{hypg}) on the event that $S_N > N$, and noting that on the event that $S_N = N$, we have $\nu_{k,N} =  X_{k,N} = 1$ for all $k$, we get
$$\P \bigg( \bigg| \frac{\nu_{k,N}}{N} - \frac{X_{k,N}}{S_N} \bigg| > \eps N \mbox{ for some }k \in [N] \bigg) \leq 2N e^{-2 \eps^2 N} \P(S_N > N).$$
Conditional on $S_N > N$, the probability that two randomly chosen individuals have the same ancestor is at least $2/(N(N+1))$, so
\begin{equation}\label{cNlower}
c_N \geq \frac{2}{N(N+1)} \P(S_N > N).
\end{equation}
The result follows.
\end{proof}

\begin{lemma}\label{nu12tail}
Suppose $\lim_{N \rightarrow \infty} c_N = 0$. For all $\eps > 0$, we have
\begin{equation}\label{tail1}
\lim_{N \rightarrow \infty} \frac{N^2}{c_N} \P \big( \nu_{1,N} \geq N \eps, \: \nu_{2,N} \geq N \eps \big) = 0.
\end{equation}
This in turn implies
\begin{equation}\label{tail2}
\lim_{N \rightarrow \infty} \frac{E[(\nu_{1,N})_2 (\nu_{2,N})_2]}{N^2 c_N} = 0.
\end{equation}
\end{lemma}

\begin{proof}
Let $\Phi_N$ be the set of all ordered pairs $(j,k)$ with $1 \leq j < k \leq N$ such that $\nu_{j,N} \geq N \eps$ and $\nu_{k,N} \geq N \eps$.  Let $|\Phi_N|$ be the cardinality of $\Phi_N$.  Note that at most $1/\eps$ of the random variables $\nu_{k,N}$ can exceed $N \eps$, and so
\begin{equation}\label{PhiNbound}
|\Phi_N| \leq 1/\eps^2.
\end{equation}
Note that
\begin{equation}\label{nuPhi}
\P \big( \nu_{1,N} \geq N \eps, \: \nu_{2,N} \geq N \eps \big) = \frac{2}{N(N-1)} E \big[ |\Phi_N| \big].
\end{equation}
Define the events
\begin{equation}\label{ANdef}
A_N = \bigg\{ \bigg| \frac{\nu_{k,N}}{N} - \frac{X_{k,N}}{S_N} \bigg| \leq \frac{\eps}{2} \mbox{ for all }k \in [N] \bigg\}
\end{equation}
and
\begin{equation}\label{BNdef}
B_N = \Big\{ S_N - \max_{1 \leq k \leq N} X_{k,N} \geq \delta a_N \Big\},
\end{equation}
where $\delta$ is the constant from Lemma \ref{iidXS}.  Also, note that
$X_{k,N}/S_N > \eps/2$ is equivalent to $X_{k,N} > (S_N - X_{k,N}) \eps/(2 - \eps)$, which on $B_N$ implies that $X_{k,N} > \eps \delta a_N/2$.  Therefore, using (\ref{PhiNbound}),
\begin{align*}
E\big[ |\Phi_N| \big] &= E \big[ |\Phi_N| \1_{A_N^c}\big] + E \big[ |\Phi_N| \1_{B_N^c}\big] + E \big[ |\Phi_N| \1_{A_N \cap B_N} \big] \\
&\leq \frac{\P(A_N^c) + \P(B_N^c)}{\eps^2} + \binom{N}{2} \P\bigg(X_{1,N} > \frac{\eps \delta a_N}{2}\bigg)^2.
\end{align*}
Now $\P(A_N^c) \ll N^{-2}c_N$ by Lemma \ref{LLNlem} and $\P(B_N^c) \ll c_N$ by part 2 of Lemma \ref{iidXS}.  Also, by (\ref{Xtail}) and (\ref{amratio}), we have $\P(X_{1,N} > \eps \delta a_N/2) \lesssim N^{-1} c_N$.  Combining these observations, and using that $c_N \rightarrow 0$ by assumption, we get
$$\lim_{N \rightarrow \infty} c_N^{-1} E\big[ |\Phi_N| \big] = 0.$$  The claim (\ref{tail1}) now follows from (\ref{nuPhi}).

To prove (\ref{tail2}), let $\eps > 0$, and note that
\begin{align*}
\frac{E[(\nu_{1,N})_2 (\nu_{2,N})_2]}{N^2 c_N} &\leq \frac{E\big[(\nu_{1,N})_2 (\nu_{2,N})_2 \1_{\{\nu_{1,N} \geq N \eps, \, \nu_{2,N} \geq N \eps\}}\big]}{N^2 c_N} \\
&\hspace{.3in}+ \frac{E\big[(\nu_{1,N})_2 (\nu_{2,N})_2 \1_{\{\nu_{1,N} < N \eps\}}\big]}{N^2 c_N} + \frac{E\big[(\nu_{1,N})_2 (\nu_{2,N})_2 \1_{\{\nu_{2,N} < N \eps\}}\big]}{N^2 c_N} \\
&\leq \frac{N^4 \P(\nu_{1,N} \geq N \eps, \, \nu_{2,N} \geq N \eps)}{N^2 c_N} + \frac{2 (N \eps) E[(\nu_{1,N})_2 \nu_{2,N}]}{N^2 c_N}.
\end{align*}
By exchangeability and the fact that $\nu_{1,N} + \dots + \nu_{N,N} = N$,
$$E[(\nu_{1,N})_2 \nu_{2,N}] = \frac{1}{N-1} E[(\nu_{1,N})_2 (\nu_{2,N} + \dots + \nu_{N,N})] \leq \frac{N}{N-1} E[(\nu_{1,N})_2] = N c_N.$$
It now follows, using (\ref{tail1}), that
$$\limsup_{N \rightarrow \infty} \frac{E[(\nu_{1,N})_2 (\nu_{2,N})_2]}{N^2 c_N} \leq \limsup_{N \rightarrow \infty} \frac{N^2}{c_N} \P(\nu_{1,N} \geq N \eps, \, \nu_{2,N} \geq N \eps) + \limsup_{N \rightarrow \infty} 2 \eps \leq 2 \eps.$$  Because $\eps > 0$ was arbitrary, the result (\ref{tail2}) follows.
\end{proof}

\begin{lemma}
Suppose $\lim_{N \rightarrow \infty} c_N = 0$.  Suppose $\Lambda$ is a probability measure on $[0,1]$.  Then, the ancestral processes $(\Psi_{n,N}(\lfloor t/c_N \rfloor))_{t \geq 0}$ in the Cannings model described above converge in $\Db_{\cP_n}[0, \infty)$ to the $\Lambda$-coalescent for all $n$ if and only if for all $x \in (0,1)$ such that $\Lambda(\{x\}) = 0$, we have
\begin{equation}\label{X1cond}
\lim_{N \rightarrow \infty} \frac{N}{c_N} \P \bigg( \frac{X_{1,N}}{a_N} > \frac{x}{1-x} \bigg) = \int_x^{1} y^{-2} \: \Lambda({\rm d}y).
\end{equation}
\end{lemma}

\begin{proof}
We need to check the three conditions of Theorem \ref{Cannings}.  Condition 1 holds by assumption, and condition 2 holds by Lemma  \ref{nu12tail}.

It remains to show that (\ref{X1cond}) is equivalent to condition 3 of Theorem \ref{Cannings}.  Define the events $A_N$ as in (\ref{ANdef}) with $\eps$ in place of $\eps/2$, and define $B_N$ as in (\ref{BNdef}).  Let $\bar{S}_{N-1} = S_N - X_{1,N}$.  Let $x \in (0, 1)$.  We have
\begin{align*}
\P(\nu_{1,N} > Nx) &\geq \P \bigg( A_N \cap \bigg\{\frac{X_{1,N}}{S_N} > x + \eps \bigg\} \bigg) \\
&\geq \P \bigg( A_N \cap \big\{ \bar{S}_{N-1} \leq (1 + \delta) a_N \big\} \cap \bigg\{ \frac{X_{1,N}}{a_N} \geq \frac{(x + \eps)(1 + \delta)}{1 - x - \eps} \bigg\} \bigg) \\
&\geq \P(\bar{S}_{N-1} \leq (1 + \delta)a_N) \P \bigg(\frac{X_{1,N}}{a_N} \geq \frac{(x + \eps)(1 + \delta)}{1 - x - \eps} \bigg) - \P(A_N^c).
\end{align*}
Note that $\P(A_N^c) \ll N^{-2} c_N$ by Lemma \ref{LLNlem}.  Also, because (\ref{Xtail}) and (\ref{amratio}) imply $X_{1,N}/a_N \rightarrow_p 0$, it follows from Lemma~\ref{iidXS} that $\bar{S}_{N-1}/a_N \rightarrow_p 1$.  Therefore,
\begin{equation}\label{infnu}
\liminf_{N \rightarrow \infty} \frac{N}{c_N} \P(\nu_{1,N} > Nx) \geq \liminf_{N \rightarrow \infty} \frac{N}{c_N} \P \bigg(\frac{X_{1,N}}{a_N} \geq \frac{(x + \eps)(1 + \delta)}{1 - x - \eps} \bigg)
\end{equation}
and
\begin{equation}\label{supX}
\limsup_{N \rightarrow \infty} \frac{N}{c_N} \P \bigg(\frac{X_{1,N}}{a_N} \geq \frac{(x + \eps)(1 + \delta)}{1 - x - \eps}\bigg) \leq \limsup_{N \rightarrow \infty} \frac{N}{c_N} \P(\nu_{1,N} > Nx).
\end{equation}

The other direction is more involved.  By exchangeability and the inclusion-exclusion formula,
\begin{equation}\label{PIE}
N \P(\nu_{k,N} > Nx) - \binom{N}{2} \P(\nu_{1,N} > Nx, \: \nu_{2,N} > Nx) \leq \P \bigg( \bigcup_{k=1}^N \{\nu_{k,N} > Nx\} \bigg).
\end{equation}
Now
\begin{align}\label{Punion}
\P \bigg( \bigcup_{k=1}^N \{\nu_{1,N} > N x\} \bigg) &\leq \P(A_N^c) + \P(B_N^c) + \P \bigg(B_N \cap \bigcup_{k=1}^N \bigg\{ \frac{X_{k,N}}{S_N} > x - \eps \bigg\} \bigg) \nonumber \\
&\leq \P(A_N^c) + \P(B_N^c) + N \P \bigg(B_N \cap \bigg\{ \frac{X_{1,N}}{S_N} > x - \eps \bigg\} \bigg).
\end{align}
Note that $\bar{S}_{N-1} \geq \delta a_N$ on $B_N$.  Then
\begin{align}\label{BXS}
&\P \bigg(B_N \cap \bigg\{ \frac{X_{1,N}}{S_N} > x - \eps \bigg\} \bigg) \nonumber \\
&\hspace{.2in}\leq \P \bigg( \big\{ \delta a_N \leq \bar{S}_{N-1} \leq (1 - \delta) a_N \big\} \cap \bigg\{ \frac{X_{1,N}}{X_{1,N} + \delta a_N} > x - \eps \bigg\} \bigg) \nonumber \\
&\hspace{1in}+ \P \bigg( \big\{ \bar{S}_{N-1} > (1 - \delta) a_N \big\} \cap \bigg\{ \frac{X_{1,N}}{X_{1,N} + (1 - \delta) a_N} > x - \eps \bigg\} \bigg) \nonumber \\
&\hspace{.2in}\leq \P\big(\bar{S}_{N-1} \leq (1 - \delta) a_N \big) \P \bigg( \frac{X_{1,N}}{a_N} \geq \frac{(x - \eps) \delta}{1 - x + \eps}\bigg) + \P \bigg( \frac{X_{1,N}}{a_N} \geq \frac{(x - \eps)(1 - \delta)}{1 - x + \eps} \bigg),
\end{align}
and plugging this back into (\ref{PIE}) and (\ref{Punion}), we get
\begin{align}\label{5terms}
N\P(\nu_{1,N} > Nx) &\leq \P(A_N^c) + \P(B_N^c) + \binom{N}{2} \P(\nu_{1,N} > Nx, \: \nu_{2,N} > Nx) \nonumber \\
&\hspace{1in}+ N\P\big(\bar{S}_{N-1} \leq (1 - \delta) a_N \big) \P \bigg( \frac{X_{1,N}}{a_N} \geq \frac{(x - \eps) \delta}{1 - x + \eps}\bigg) \nonumber \\
&\hspace{2.3in}+ N\P \bigg( \frac{X_{1,N}}{a_N} \geq \frac{(x - \eps)(1 - \delta)}{1 - x + \eps} \bigg).
\end{align}
We now show that the first four terms on the right-hand side of (\ref{5terms}) are small.  We have $\P(A_N^c) \ll N^{-2}c_N$ by Lemma \ref{LLNlem}, and $\P(B_N^c) \ll c_N$ by part 2 of Lemma \ref{iidXS}.  We have $\P(\nu_{1,N} > Nx, \: \nu_{2,N} > Nx) \ll N^{-2} c_N$ by Lemma \ref{nu12tail}.  Recall also that $\bar{S}_{N-1}/a_N \rightarrow_p 1$, and therefore $\P(\bar{S}_{N-1} \leq (1 - \delta) a_N) \rightarrow 0$, while we have
$\P(X_{1,N} \geq \theta a_N) \lesssim N^{-1} c_N$ for all $\theta > 0$ by (\ref{Xtail}) and (\ref{amratio}).
Thus, we get
\begin{equation}\label{supnu}
\limsup_{N \rightarrow \infty} \frac{N}{c_N} \P(\nu_{1,N} > Nx) \leq \limsup_{N \rightarrow \infty} \frac{N}{c_N} \P \bigg( \frac{X_{1,N}}{a_N} \geq \frac{(x - \eps)(1 - \delta)}{1 - x + \eps} \bigg)
\end{equation}
and
\begin{equation}\label{infX}
\liminf_{N \rightarrow \infty} \frac{N}{c_N} \P \bigg( \frac{X_{1,N}}{a_N} \geq \frac{(x - \eps)(1 - \delta)}{1 - x + \eps} \bigg) \geq \liminf_{N \rightarrow \infty} \frac{N}{c_N} \P(\nu_{1,N} > Nx).
\end{equation}

Recall that we need to show that condition 3 of Theorem~\ref{Cannings} is equivalent to (\ref{X1cond}).  First, suppose condition 3 of Theorem~\ref{Cannings} holds.  Choose $x \in (0, 1)$ such that $\Lambda(\{x\}) = 0$.  Because $\delta$ and $\eps$ can be arbitrarily small, equations (\ref{supX}) and (\ref{infX}) imply that for all $\theta > 0$ for which $\Lambda(\{x - \theta\}) = \Lambda(\{x + \theta\}) = 0$, we have
\begin{align*}
\int_{x + \theta}^1 y^{-2} \: \Lambda({\rm d}y) &= \liminf_{N \rightarrow \infty} \frac{N}{c_N}\P\big(\nu_{1,N} > N(x + \theta)\big) \leq \liminf_{N \rightarrow \infty} \frac{N}{c_N} \P \bigg( \frac{X_{1,N}}{a_N} \geq \frac{x}{1-x} \bigg) \\
&\leq \limsup_{N \rightarrow \infty} \frac{N}{c_N} \P \bigg( \frac{X_{1,N}}{a_N} \geq \frac{x}{1-x} \bigg) \leq \limsup_{N \rightarrow \infty} \P\big(\nu_{1,N} > N(x - \theta)\big) \\
&= \int_{x - \theta}^{1} y^{-2} \: \Lambda({\rm d}y).
\end{align*}
Because $\Lambda(\{x\}) = 0$, equation (\ref{X1cond}) follows.  Conversely, suppose (\ref{X1cond}) holds for all $x$ such that $\Lambda(\{x\}) = 0$.  Reasoning as above, equations (\ref{infnu}) and (\ref{supnu}) imply that for all $\theta > 0$ for which $\Lambda(\{x - \theta\}) = \Lambda(\{x + \theta\}) = 0$, we have
\begin{align*}
\int_{x+\theta}^1 y^{-2} \: \Lambda({\rm d}y) &= \liminf_{N \rightarrow \infty} \frac{N}{c_N} \P \bigg( \frac{X_{1,N}}{a_N} > \frac{x + \theta}{1 - (x + \theta)} \bigg) \leq \liminf_{N \rightarrow \infty} \frac{N}{c_N} \P(\nu_{1,N} > Nx) \\
&\leq \limsup_{N \rightarrow \infty} \frac{N}{c_N} \P(\nu_{1,N} > Nx) \leq \limsup_{N \rightarrow \infty} \frac{N}{c_N} \P \bigg( \frac{X_{1,N}}{a_N} \geq \frac{x - \theta}{1 - (x - \theta)} \bigg) \\
&= \int_{x - \theta}^1 y^{-1} \: \Lambda({\rm d}y).
\end{align*}
Therefore, condition 3 of Theorem~\ref{Cannings} holds, and the proof is complete.
\end{proof}

\section{Results for two-point distributions}\label{2ptsec}

We would like to prove the results for the simple model with a two-point distribution for the activation times introduced in Section \ref{ssec:simple}. To this end, we begin with an observation for the asymptotic behaviour of the moments of the geometric distributions that will govern the numbers of offspring of the early bird in each of the three regimes. The more general results from Section~\ref{sec:cannings} will be very useful.

Let $G_{1,N}$ be a geometric random variable with parameter $e^{-\lambda_N T_N}$. The following lemma provides, under condition \eqref{cond}, the asymptotic behavior of the $n$th moments 
$$M_N(n)\coloneqq  E\left[\left(\frac{G_{1,N}}{G_{1,N}+N-1}\right)^n\right].$$
\begin{lemma}\label{moments}
Let $n\in\Nb$. Assume that \eqref{cond} holds with $\kappa,\beta>0$. 
\begin{enumerate}
 \item If $\beta>1$, $\lim_{N\to\infty}M_N(n)=1$.
 \item If $\beta=1$, $\lim_{N\to\infty}M_N(n)= E[Y_\kappa^n]$, where $Y_\kappa$ is a random variable on $[0,1]$ with distribution
$$\P(Y_\kappa>x)=e^{- \frac{x}{\kappa(1-x)}},\qquad x\in[0,1].$$
 \item If $\beta<1$, $M_N(n)\sim n! \kappa^{n\beta}N^{-n(1-\beta)}$ as $N\to\infty$.
\end{enumerate}
\end{lemma}

\begin{proof}
Note first that
 \begin{equation}\label{i3}
 M_N(n)=\int_{0}^1\P\Bigg( \frac{G_{1,N}}{G_{1,N}+N-1}> x^{1/n}\Bigg){\rm d} x=\int_{0}^1\P\Bigg( G_{1,N}>\frac{(N-1)x^{1/n}}{1-x^{1/n}} \Bigg){\rm d} x.
\end{equation}
Assume now that $\beta>1$ and let us show part 1 of the lemma. Note that, for any $y\geq 0$
$$\P\left( G_{1,N}> (N-1)y \right)=\left (1-\frac{1}{(\kappa N)^{\beta}}\right)^{\lfloor (N-1)y\rfloor}\xrightarrow[N\to\infty]{}1.$$
Thus, part 1 of the lemma follows from the dominated convergence theorem. 

Assume now that $\beta=1$ and let us prove part 2. Note that, for any $y\geq 0$
$$\P\left( G_{1,N}> (N-1)y \right)=\left (1-\frac{1}{\kappa N}\right)^{\lfloor (N-1)y\rfloor}\xrightarrow[N\to\infty]{}e^{-\frac{y}{\kappa}}.$$
Thus, part 2 follows using dominated convergence theorem and making the substitution $y=x^{1/n}$. 

In the remainder of the proof we assume that $\beta<1$. Making the change of variable $y=x^{1/n}/(1-x^{1/n})$ in \eqref{i3} and using standard properties of the floor function, we obtain
$$ M_N(n)\sim \int_{0}^{\infty}\frac{ny^{n-1}}{(1+y)^{n+1}}\left(1-\frac{1}{(\kappa N)^\beta}\right)^{(N-1)y}{\rm d} y= \int_{0}^{\infty}\frac{ny^{n-1}}{(1+y)^{n+1}}e^{-\alpha_N y}{\rm d} y,$$
as $N\to \infty$, where $\alpha_N=-(N-1)\log(1-(\kappa N)^{-\beta})$. Making the substitution $z=\alpha_N y$, we obtain
$$M_N(n)\sim\alpha_N^{-n}\int_0^\infty \frac{nz^{n-1}}{\left(1+\frac{z}{\alpha_N}\right)^{n+1}}e^{-z} {\rm d} z\sim\alpha_N^{-n}\int_0^\infty {nz^{n-1}}e^{-z} {\rm d} z=\alpha_N^{-n}n!,$$
and part 3 follows since $\alpha_N\sim \kappa^{-\beta} N^{1-\beta}$ as $N\to \infty$.
\end{proof}

\begin{lemma}\label{c1-J}
Under assumption \eqref{cond0}, we have
$$\lim_{N\to\infty} c_N=0.$$
Moreover, if \eqref{cond} holds for
\begin{enumerate}
 \item  $\beta>1$, then $c_N\sim N\om_N$ as $N\to\infty$.
 \item  $\beta=1$, then $c_N\sim N\om_N \,E[Y_\kappa^2]$ as $N\to\infty$.
 \item  $\beta<1$, then $c_N\sim 2\kappa ^{2\beta} \om_N N^{2\beta-1}$ as $N\to\infty$.
\end{enumerate}
\end{lemma}
\begin{proof}
We begin with a few general observations. Let $\bar{S}_{N-1} = S_N - X_{1,N}=\sum_{i=2}^N X_{i,N}$, and note that, thanks to Lemma \ref{dict}, we have
\begin{equation}\label{i1}
c_N=NE\left[\frac{(X_{1,N})_2}{(S_N)_2}\right]=N\omega_N E\left[\frac{(G_{1,N})_2}{(G_{1,N}+\bar{S}_{N-1})_2}\right],
\end{equation}
Since the expectation in the previous expression is smaller than one, the first statement follows from the assumption that $N\omega_N\to 0$ as $N\to \infty$.

Let us now have a closer look at the expectation in \eqref{i1}. Splitting on the event that there exists $b_{i,N}=1$ with $i>1$ and its complement, and defining
$$I_N^+\coloneqq E\left[\frac{(G_{1,N})_2}{(G_{1,N}+\bar{S}_{N-1})_2}\;\bigg|\; \sum_{i=2}^N b_{i,N} >0\right]\leq 1\quad\textrm{and}\quad I_N^0 \coloneqq E\left[\frac{(G_{1,N})_2}{(G_{1,N}+{N-1})_2}\right],$$
we obtain
\begin{equation}\label{i2}
 E\left[\frac{(G_{1,N})_2}{(G_{1,N}+\bar{S}_{N-1})_2}\right]=(1-(1-\omega_N)^{N-1})I_N^+ + (1-\omega_N)^{N-1} I_N^0,
\end{equation}
Clearly, $0\leq I_N^+\leq I_N^0$, and because $N\omega_N\to 0$ it follows that $I_N^0$ is the leading term in \eqref{i2} and 
\begin{equation}\label{eq:cN_only_1_matters}
 c_N= N\omega_N E\left[\frac{(G_{1,N})_2}{(G_{1,N}+\bar{S}_{N-1})_2}\right]\sim N\omega_N I^0_N, \quad \text{ as } N\rightarrow \infty.
\end{equation}
In addition, we have
\begin{equation}\label{i2b}
 I_N^0=M_N(2)-E\left[\frac{(N-1)G_{1,N}}{(G_{1,N}+N-1)(G_{1,N}+N-1)_2}\right],
\end{equation}
and the second term is smaller than $1/N$. Note that we have not used any assumptions on the distribution of $G_{1,N}$ up to this point. 

Parts 1 and 2 of the lemma now follow directly using Lemma \ref{moments}. In the remainder of the proof we assume that \eqref{cond} holds for $\beta<1$. In order to prove part 3, we use Lemma \ref{moments} to see that
$$M_N(2)\sim\frac{2\kappa^{2\beta}}{N^{2(1-\beta)}},$$
as $N\to \infty$. To complete the proof, we need to show that the second term in \eqref{i2b} converges faster to $0$. Note that, using Lemma \ref{moments} with $n=1$ in the last step, we get
$$E\left[\frac{G_{1,N}}{(G_{1,N}+N-1)^2}\right]\leq \frac{1}{N}M_N(1)\sim \frac{\kappa^{\beta}}{N^{2-\beta}},$$
as $N\to\infty$.  This completes the proof. 
\end{proof}

\begin{lemma}\label{c-King}
Assume that condition \eqref{cond0} holds. If in addition $\eqref{cond}$ holds for $\beta<1$,
then $$\lim_{N\to\infty}\frac{E[(\nu_{1,N})_3]}{N^2 c_N}=0.$$
In particular, in this case, the processes $(\Psi_{n,N}(t/c_N))_{t\geq 0}$ converge to Kingman's coalescent.
\end{lemma}
\begin{proof}
The second statement follows from the first one using \cite[Thm. 4(b)]{mo00}.
 Let us prove the first statement. Thanks to Lemma \ref{dict}, and using that $(X_{1,N})_3=0$ if $b_{1,N}=0$, we obtain
 $$\frac{E[(\nu_{1,N})_3]}{N^2 c_N}=\frac{(N-1)(N-2)}{N c_N}E\left[\frac{(X_{1,N})_3}{(S_N)_3}\right]\leq\frac{N\omega_N}{c_N}E\left[\frac{(G_{1,N})_3}{(G_{1,N}+\bar{S}_{N-1})_3}\right],$$
 with $\bar{S}_{N-1}=\sum_{k=2}^N X_{i,N}$.
Moreover, from part 3 in Lemma \ref{c1-J}, we see that it suffices to show that
$$\lim_{N\to\infty}N^{2(1-\beta)}E\left[\frac{(G_{1,N})_3}{(G_{1,N}+
\bar{S}_{N-1})_3}\right]=0.$$
Now, using that $\bar{S}_{N-1}\geq N-1$ and that, for $a\leq b$ and $b\geq 3$, we have $(a)_3/(b)_3\leq (a/b)^3$, we obtain
\begin{align*}
 E\left[\frac{(G_{1,N})_3}{(G_{1,N}+\bar{S}_{N-1})_3}\right]\leq E\left[\left(\frac{G_{1,N}}{G_{1,N}+{N-1}}\right)^3\right]=M_N(3),
\end{align*}
for $N\geq 3$, and the result follows using Lemma \ref{moments} with $n=3$.
 \end{proof}
 
\begin{proof}[Proof of Theorem \ref{simplethm}]
 The case $\beta <1$ is already covered by Lemma \ref{c-King}.
 
 We want to apply Theorem \ref{Cannings} for the case of $\beta \geq 1$. Conditions 1 and 2 hold by Lemma~\ref{c1-J} and Lemma~\ref{nu12tail} respectively. Hence we are left to check condition 3.
 
 Using Lemma \ref{LLNlem}, we obtain that for any $x \in (0,1)$ and any $\varepsilon \in (0, x\wedge(1-x))$, for sufficiently large $N$
 \begin{align}\label{eq:nu_vs_X}
\frac{N}{c_N}\P\left(\frac{X_{1,N}}{S_N} > x+\varepsilon\right) \leq  \frac{N}{c_N}\P\left(\nu_{1,N} > Nx\right) \leq  \frac{N}{c_N}\P\left(\frac{X_{1,N}}{S_N} > x-\varepsilon\right).
 \end{align}
 Let $A$ be the event that the individual with label 1 woke up early, i.e.\ $A := \{ \tau_{1,N} = 0\}=\{ b_{1,N} = 1\}$, and let $B$ be the event that at least two individuals woke up early, i.e.\ $B:=\{\sum_{i=1}^N b_{i,N} =2\}$. 
 For any $y \in (0,1)$ we can split the probability we are interested in into
 \begin{align*}
 \P\left(\frac{X_{1,N}}{S_N} > y\right)	& = \P\left(\frac{X_{1,N}}{S_N} > y\mid A^c\right)\P(A^c)  +\P\left(\frac{X_{1,N}}{S_N} > y\mid A\cap B^c\right)\P\left( A\cap B^c\right) \\								       	       & \qquad \qquad \qquad \qquad +\P\left(\frac{X_{1,N}}{S_N} > y\mid A\cap B\right)\P\left( A\cap B\right).
 \end{align*}
 Note that $\P\left(X_{1,N}/S_N > y\mid A^c\right) \leq \P(1/N>y)=0$ for  any given $y$ and $N$ sufficiently large. In addition we can bound $\P\left( A\cap B\right)\leq N\omega_N^2$. By Lemma \ref{c1-J}, we have $c_N \sim N \omega_N C(\beta)$, where $C(\beta)=1$ if $\beta >1$ and $C(1) = \E[Y_{\kappa}^2]$ if $\beta=1$.  Hence,
 \begin{align*}
  \limsup_{N\rightarrow \infty} \frac{N}{c_N}\P\left(\frac{X_{1,N}}{S_N} > y\mid A\cap B\right)\P\left( A\cap B\right)=0.
 \end{align*}
Observe that we can rewrite
\begin{align*}
\P\left(\frac{X_{1,N}}{S_N} > y\mid A\cap B^c\right) 	& = \P\left(\frac{G_{1,N}}{G_{1,N} + N-1} > y\right) = \P\left(G_{1,N}>\frac{y}{1-y}(N-1)\right) \\
								& = \left(1-\frac{1}{(\kappa N)^{\beta}}\right)^{\left\lfloor\frac{y}{1-y}(N-1)\right\rfloor},
\end{align*}
 since $G_{1,N}$ has a geometric distribution with parameter $(\kappa \beta)^{-1}$.
 
 If $\beta > 1$, combining the observations above, we obtain
 \begin{align*}
  \lim_{N\rightarrow \infty} \frac{N}{c_N}\P\left(\frac{X_{1,N}}{S_N} > y\right) & = \lim_{N\rightarrow \infty}\frac{N\omega_N(1-\omega_N)^{N-1}}{N\omega_N}\left(1-\frac{1}{(\kappa N)^{\beta}}\right)^{\left\lfloor\frac{y}{1-y}(N-1)\right\rfloor}  = 1 
 \end{align*}
 for any $y \in (0,1)$. Setting $\Lambda_{ss} := \delta_1$, the Dirac-measure on 1, we conclude from \eqref{eq:nu_vs_X} that 
 \begin{align*}
   \lim_{N\rightarrow \infty} \frac{N}{c_N}\P\left(\nu_{1,N} > Nx\right) & = 1 = \int_x^1z^2\Lambda_{ss}({\rm d}z),
										 \end{align*}
for any $x \in (0,1)$. Theorem \ref{Cannings} then implies that the genealogy converges to the star-shaped coalescent.
 
Analogously, if $\beta=1$, then defining $Y_{\kappa}$ as in part 2 of Theorem \ref{simplethm},
$$\lim_{N\rightarrow \infty} \frac{N}{c_N}\P\left(\frac{X_{1,N}}{S_N} > y \right) = \lim_{N\rightarrow \infty}\frac{N\omega_N(1-\omega_N)^{N-1} }{c_N}\left(1-\frac{1}{\kappa N}\right)^{\left\lfloor\frac{y}{1-y}(N-1)\right\rfloor}  = \frac{1}{\E[Y_\kappa^2]} e^{-\frac{1}{\kappa}\frac{y}{1-y}}.$$
Since \eqref{eq:nu_vs_X} holds for all $\varepsilon >0$ and $\sup_{\varepsilon >0} \frac{x-\varepsilon}{1-(x-\varepsilon)} = \inf_{\varepsilon >0} \frac{x+\varepsilon}{1-(x+\varepsilon)} =  \frac{x}{1-x}$ this implies that 
\begin{align*}
\lim_{N\rightarrow \infty} \frac{N}{c_N}\P\left(\nu_{1,N} > Nx\right) & = \frac{1}{\E[Y_\kappa^2]} e^{-\frac{1}{\kappa}\frac{x}{1-x}} =  \frac{1}{\E[Y_\kappa^2]}\P(Y_\kappa>x)
\end{align*}
Hence, setting $\Lambda_\kappa({\rm d}y) := \frac{y^2}{\E[Y_\kappa^2]}\P(Y_\kappa\in{\rm d}y)$, we obtain
\begin{align*}
\lim_{N\rightarrow \infty} \frac{N}{c_N}\P\left(\nu_{1,N} > Nx\right) & = \int_x^1y^{-2} \Lambda_{\kappa}({\rm d}y)
\end{align*}
for all $x \in (0,1)$, which ends the proof.
\end{proof}

\section{Exponentially increasing rates of exit from dormancy}\label{expsec}
 
This section is devoted to the proof of Theorem \ref{expthm}, which characterizes the asymptotic genealogies in the model described in Section \ref{expsub}.

\subsection{A comparison between the genealogies of two models}

The main ingredient in the proof of Theorem \ref{expthm} is a result that allows us to compare the genealogies of two populations constructed from the same sequence $(\zeta_i)_{i=1}^N$ of i.i.d. positive random variables. The first model is the one described in Section \ref{modelsec}, with no summer (i.e. $t_N=T_N$), with $\lambda_N=\lambda>0$, and where $T_N-\tau_{i,N}= \zeta_i\wedge T_N$ (the model in Section \ref{expsub} is the special case where $\zeta_1$ is exponentially distributed with parameter $\gamma$). In the second model, the family sizes $(\tilde{X}_1,\ldots,\tilde{X}_N)$ at the end of the year are i.i.d. and such that, conditionally on $\zeta_i$, $\tilde{X}_i$ is geometrically distributed with parameter $e^{-\lambda \zeta_i}$ (i.e. days start at time $-\infty$). The vector of family sizes $(\tilde{\nu}_{1,N},\ldots,\tilde{\nu}_{N,N})$ is obtained by sampling $N$ individuals without replacement among the $\tilde{S}_N:=\tilde{X}_1+\cdots +\tilde{X}_N$ present at the end of the year.

The next result will be useful for comparing the genealogies of the previously described models. 
\begin{lemma}\label{coupling}
Let $f:\Rb_+^N\to \Rb_+$ be a positive bounded function. Then,
$$\lvert E[f(X_{1,N},\ldots,X_{N,N})]- E[f(\tilde{X}_1,\ldots,\tilde{X}_N)]\rvert \leq 2 \lVert f\rVert_{\infty} (1-\P(\zeta_1\leq T_N)^N).$$
\end{lemma}
\begin{proof}
It follows directly from the fact that $(X_{1,N},\ldots,X_{N,N})$ and $(\tilde{X}_1,\ldots,\tilde{X}_N)$ are equal if $\zeta_i\leq T_N$ for all $i\in[N]$.
\end{proof}
Let us now set
$$ c_N:=\frac{ E[(\nu_{1,N})_2]}{N-1}\quad\textrm{and}\quad \tilde{c}_N:=\frac{ E[(\tilde{\nu}_{1,N})_2]}{N-1}.$$
The next result provides sufficient conditions for the limiting genealogies of the two models to coincide.
\begin{prop}\label{equiva}
 Assume that 
 \begin{equation}\label{condTN}
  \lim_{N\to\infty}N^3\P(\zeta_1> T_N)=0.
 \end{equation}
Then $c_N\sim \tilde{c}_N$ as $N\to\infty$. Moreover, for $r\geq 1$, the condition
 \begin{equation}\label{condTNb}
  \lim_{N\to\infty}N^{r+2}\P(\zeta_1> T_N)=0
 \end{equation}
implies that, for all $k_1,\ldots,k_r\geq 2$,
$$\lim_{N\to\infty}\frac{ E[(\nu_{1,N})_{k_1}\cdots(\nu_{r,N})_{k_r}]}{N^{k_1+\cdots+k_r-r}c_N}= \lim_{N\to\infty}\frac{ E[(\tilde{\nu}_{1,N})_{k_1}\cdots(\tilde{\nu}_{r,N})_{k_r}]}{N^{k_1+\cdots+k_r-r}\tilde{c}_N},$$
in the sense that if either limit exists, then so does the other, and the limits are equal.
\end{prop}

\begin{proof}
Using Lemma \ref{dict}  for the two models we obtain
 $$c_N=N E\left[\frac{(X_{1,N})_2}{(S_{N})_2}\right]\quad\textrm{and}\quad\tilde{c}_N=N E\left[\frac{(X_1)_2}{(\tilde{S}_{N})_2}\right].$$
 Thus, using Lemma \ref{coupling} with $f$ defined via $f(x_1,\ldots,x_N)=N(x_1)_2/(x_1+\cdots +x_N)_2$, we obtain
 $$|c_N-\tilde{c}_N|\leq 2 N (1-\P(\zeta_1<T_N)^N),$$
and the first result follows from \eqref{condTN} and (\ref{lowcN}).

For the second statement, let us assume that \eqref{condTNb} holds. Note first that Lemma \ref{dict} yields
\begin{equation}\label{fid}
 \frac{ E[(\tilde{\nu}_{1,N})_{k_1}\cdots(\tilde{\nu}_{r,N})_{k_r}]}{N^{k_1+\cdots+k_r-r}\tilde{c}_N}=\frac{(N)_{k_1+\ldots+k_r}}{N^{k_1+\ldots+k_r}}\frac{N^r}{\tilde{c}_N} E\left[\frac{(\tilde{X}_1)_{k_1}\cdots(\tilde{X}_r)_{k_r}}{(\tilde{S}_N)_{k_1+\cdots+k_r}}\right],
\end{equation}
and 
\begin{equation}\label{sid}
 \frac{ E[(\nu_{1,N})_{k_1}\cdots({\nu}_{r,N})_{k_r}]}{N^{k_1+\cdots+k_r-r}c_N}=\frac{(N)_{k_1+\ldots+k_r}}{N^{k_1+\ldots+k_r}}\frac{N^r}{c_N} E\left[\frac{(X_{1,N})_{k_1}\cdots(X_{r,N})_{k_r}}{(S_N)_{k_1+\cdots+k_r}}\right].
\end{equation}
Consider the function $f:\Nb^N\to\Rb_+$ defined via $$f(x_1,\ldots,x_N)=\frac{(x_1)_{k_1}\cdots(x_r)_{k_r}}{(x_1+\cdots +x_N)_{k_1+\cdots+k_r}}.$$
Note that $f(x_1,\ldots,x_N)$ gives the probability that, if we sample $k_1+\cdots+k_r$ balls from an urn containing, for each $i\in[N]$, $x_i$ balls with label $i$, the first $k_1$ of them have label $1$, the next $k_2$ balls have label $2$, and so on. In particular, $\lVert f\Vert_\infty\leq 1$. Thus, applying Lemma \ref{coupling} with $f$,  
we obtain, for $N$ sufficiently large,
\begin{equation}\label{difference}
 \frac{N^r}{\bar{c}_N}\left|  E\left[\frac{(\bar{X}_1)_{k_1}\cdots(\bar{X}_r)_{k_r}}{(\bar{S}_N)_{k_1+\cdots+k_r}}\right]-E\left[\frac{(X_{1,N})_{k_1}\cdots(X_{r,N})_{k_r}}{(S_N)_{k_1+\cdots+k_r}}\right]\right|\leq \frac{2 N^r}{\bar{c}_N}(1-\P(\zeta_1<T_N)^N).
\end{equation}
Therefore, using (\ref{lowcN}) and \eqref{condTNb}, we conclude that the left-hand side in \eqref{difference}
converges to zero as $N\to\infty$. Moreover, since \eqref{condTNb} implies \eqref{condTN}, we have $c_N\sim\tilde{c}_N$. Hence, the result follows from \eqref{fid} and \eqref{sid}.
\end{proof}

\subsection{Exponential model} 
In this section, we come back to the model described in Section \ref{expsec}, i.e. we assume that $\zeta_1$ satisfies \eqref{zetatail}, that its law has no mass at zero, and that $T_N$ satisfies \eqref{zetaTN}.

The next result provides the asymptotic behaviour of the tails of $\tilde{X}_1$.
\begin{lemma}\label{l-tail}
 Assume that \eqref{zetatail} holds, and set $a=\gamma/\lambda$. Then,
\begin{equation}\label{jc}
 \P(\tilde{X}_1>k)\sim c\,\Gamma(1+a) k^{-a},\quad \textrm{as $k\to\infty$}.
\end{equation}
\end{lemma}
\begin{proof}
 Let $\mu$ denote the law of $\zeta_1$. Consider the bijective function $f_\lambda:(0,\infty) \to (0,\infty)$ defined via
 $$f_\lambda(u):=-\ln(1-e^{-\lambda u}),\quad u\in (0,\infty),$$
 and the measure $\mu_\lambda:=\mu\circ f_\lambda^{-1}$, i.e. $\mu_\lambda$ is the push forward measure of $\mu$ by $f_\lambda$. With these definitions, we have
 $$\P(\tilde{X}_1>k)=\int_0^{\infty}(1-e^{-\lambda u})^k\P(\zeta_1\in {\rm d} u)=\int_0^\infty e^{-k u} \mu_\lambda ({\rm d}u),$$
 i.e. the tail of $\tilde{X}_1$ is the Laplace transform of the measure $\mu_\lambda$. The asymptotic behaviour of this Laplace transform around $\infty$ relates to the behaviour around zero of the distribution function of $\mu_\lambda$ via a Tauberian theorem. Let us make this precise. Set $\bar{\mu}(x)\coloneqq \mu((y,\infty))$, $y>0$, and note that
 $$\mu_{\lambda}((0,x])=\bar{\mu}(f_\lambda^{-1}(x)).$$
 Hence, using the substitution $y=f_\lambda^{-1}(x)$, we obtain
 \begin{align*}
  \lim_{x\to 0+} \frac{\mu_{\lambda}((0,x])}{x^{a}}=\lim_{y\to\infty} \frac{\bar{\mu}(y)}{(f_\lambda(y))^a}=\lim_{y\to\infty} \frac{\bar{\mu}(y)}{e^{-\lambda a y}}=c.
 \end{align*}
The result then follows directly from Karamata's Tauberian theorem (Theorem 1.7.1 of \cite{bgt}).
\end{proof}

\begin{proof}[Proof of Theorem \ref{expthm}]
Let us denote by $(\tilde{\Psi}_{n,N}(d))_{d=0}^\infty$ the ancestral process associated with the model whose family sizes (before sampling) are $(\tilde{X}_1,\ldots,\tilde{X}_N)$, and set $a\coloneqq\gamma/\lambda$. 

Since the law of $\zeta_1$ satisfies \eqref{zetatail}, we conclude from Lemma \ref{l-tail} that the tails of $\tilde{X}_1$ satisfy \eqref{jc}. Hence, Theorem 4 in \cite{sch03} implies that: i) if $a\geq 1$, the processes $(\tilde{\Psi}_{n,N}(\lfloor t/\tilde{c}_N\rfloor))_{t\geq 0}$ converge as $N\to\infty$ to the $\Lambda$-coalescent in which $\Lambda$ is $\delta_0$ if $a\geq 2$ or the Beta$(2-a,a)$ distribution if $a\in[1,2)$, and ii) if $a\in(0,1)$, the processes $(\tilde{\Psi}_{n,N}(d))_{d=0}^\infty$ converge as $N\to\infty$ to the $\Xi_a$-coalescent described in Theorem 4(d) of \cite{sch03}. 

It remains to prove that that the processes $({\Psi}_{n,N}(d))_{d=0}^\infty$ and $(\tilde{\Psi}_{n,N}(d))_{d=0}^\infty$, in the appropriate time scale, have the same limiting genealogy. For this, note that, under \eqref{zetatail} and \eqref{zetaTN}, conditions \eqref{condTN} and \eqref{condTNb} are satisfied. 
Hence, the result follows combining Proposition \ref{equiva} with Theorem~2.1 in \cite{mosa01}.
\end{proof}

\section{Spring and summer}\label{summersec}

In this section we prove Proposition \ref{springsummer}, which, informally speaking, states that under mild hypotheses, the length of the summer $T_N-t_N$ does not change the genealogy of the population.  The key to the argument will be the following lemma about Polya urns.

\begin{lemma}\label{polya}
Consider a Polya urn in which there are initially $M$ balls, all of different colors.  We repeatedly draw a ball at random from the urn and return it to the urn along with another ball of the same color, until there are $N$ balls in the urn.  Then if we choose two balls at random from the urn, the probability that they are the same color is at most $2/(M+1)$.
\end{lemma}

\begin{proof}
Let $D$ be the event that both balls have the same color.  For $k \in \{0, 1, 2\}$, let $R_k$ be the event that $k$ out of the two balls chosen are among the $M$ balls that were in the urn at the beginning.  Then $P(R_2|D) = 0$ because the initial $M$ balls all have different colors.  By symmetry, we have $P(R_1|D) = 1/M$.  Finally, the well-known exchangeability of Polya urns implies that $P(R_0|D)$ is the same as the probability that the first two balls added to the urn are the same color, which is $2/(M+1)$.  It follows that $P(D) = \sum_{k=0}^2 P(R_k)P(D|R_k) \leq 2/(M+1)$.
\end{proof}

\begin{proof}[Proof Proposition \ref{springsummer}]
For the population with a summer, we want to bound the probability $q_N$ that two individuals chosen at random at time $T_N$ have the same ancestor at time $t_N$.  Let $M = Y_{1,N} + \dots + Y_{N,N}$ be the number of individuals alive at time $t_N$.  Between times $t_N$ and $T_N$, the sizes of the families started by the $M$ individuals at time $t_N$ evolve as independent Yule processes.  Using the a well-known connection between Yule processes and Polya urns, one can see that the sequence keeping track of the families into which successive individuals are born follows the same dynamics as the sequence keeping track of the colors of successive balls added to a Polya urn starting from $M$ balls of different colors.  Thus, by Lemma \ref{polya}, the probability, conditional on $M$, that two individuals chosen at random at time $T_N$ have the same ancestor at time $t_N$ is at most $2/(M+1)$, which means the unconditional probability satisfies $$q_N \leq 2 E \bigg[ \frac{1}{Y_{1,N} + \dots + Y_{N,N}} \bigg].$$

Consider the canonical coupling of $\Psi_{n,N}$ and $\hat \Psi_{n,N}$ using the same activation times $\tau_{i,N}^{(d)}$ and the same birth times during the spring for both processes.  Let $D_N$ be the first day (or year) that at least one pair of individuals in the sample in $\Psi_{n,N}$ finds a common ancestor during the summer, that is, in $(t_N,T_N]$. By the coupling, the processes $\Psi_{n,N}$ and $\hat \Psi_{n,N}$ coincide until time $D_N$, and by the observations above and condition \eqref{H}, for all $K > 0$ we have
 \begin{align*}
&\P\big(\Psi_{n,N}(d) = \hat \Psi_{n,N}(d) \:\: \forall d = 0, \ldots, \lfloor K\rho_N\rfloor \big) \\
&\hspace{1in}\geq  \P(D_N > K\rho_N) \geq \left(1- 2 \binom{n}{2} \E\left[\frac{1}{Y_{1,N}+\cdots+Y_{N,N}}\right]\right)^{K\rho_N} \xrightarrow{N\rightarrow \infty} 1,
 \end{align*}
which implies the result.
\end{proof}

\section{Classifying the possible limits}\label{classsec}

In this section, we prove Theorem \ref{charthm}, which classifies all possible $\Lambda$-coalescents that can arise as limits in the model introduced in Section \ref{modelsec}.  

\begin{rem}\label{probrem}
Suppose the ancestral processes $(\Psi_{n,N}(\lfloor \rho_N t \rfloor))_{t \geq 0}$ converge to the $\Lambda$-coalescent for all $n$, where $\Lambda$ is a finite nonzero measure.  Because $\Lambda([0,1])$ is the rate at which two randomly chosen lineages merge in the $\Lambda$-coalescent and $c_N$ is the probability that two individuals have the same parent in one generation in the Cannings model, we must have $\rho_N c_N \sim \Lambda([0,1])$ as $N \rightarrow \infty$.  It follows that if we replace $\rho_N$ by $1/c_N$, then all of the transition rates in the limit will be multiplied by $1/\Lambda([0,1])$.  Therefore, the ancestral processes $(\Psi_{n,N}(\lfloor t/c_N \rfloor))_{t \geq 0}$ will converge for all $n$ to the $\tilde{\Lambda}$-coalescent, where $\tilde{\Lambda} = \Lambda/\Lambda([0,1])$ is a probability measure.

Conversely, suppose the ancestral processes $(\Psi_{n,N}(\lfloor t/c_N \rfloor))_{t \geq 0}$ converge for all $n$ to the $\Lambda$-coalescent, where $\Lambda$ is a probability measure.  Let $a > 0$.  Then, if we choose $\rho_N = a/c_N$, the ancestral processes $(\Psi_{n,N}(\lfloor \rho_N t \rfloor))_{t \geq 0}$ will converge for all $n$ to the $a \Lambda$-coalescent.  We can also, of course, obtain convergence to the zero measure by choosing $\rho_N$ such that $\rho_N c_N \rightarrow 0$ as $N \rightarrow \infty$.
\end{rem}

In view of Remark \ref{probrem}, we may restrict our attention to the case in which $\rho_N = 1/c_N$, and to show that the possible $\Lambda$-coalescents that can arise as limits with this choice of scaling are precisely the probability measures satisfying (\ref{decomp}) and (\ref{hdef}).  Proposition \ref{construction} shows how to obtain all of these $\Lambda$-coalescents as limits.

\begin{prop}\label{construction}
 Let $\Lambda$ be a probability measure on $[0,1]$ of the form
 \begin{align}\label{eq:Lambda}
  \Lambda = a_1\delta_0 + a_2\delta_1 + \Lambda'
 \end{align}
 for nonnegative real numbers $a_1$ and $a_2$ and $\Lambda'$ a measure on $(0,1)$ with density $h$ given by \eqref{hdef}. Then there exist choices for birth-rates $(\lambda_N)_{N=1}^{\infty}$, the times $(t_N)_{N=1}^{\infty}$, $(T_N)_{N=1}^{\infty}$ distributions of the wake-up times $(\tau_{1,N})_{N=1}^{\infty}$ such that the ancestral processes $(\Psi_{n,N}(\lfloor t/c_N \rfloor))_{t \geq 0}$ converge for all $n$ to the $\Lambda$-coalescent. 
\end{prop}

\begin{proof}
We will construct the approximating Cannings model as a mixture of the simple two-point models discussed in Section \ref{ssec:simple} and deduce the convergence of their ancestral processes to the desired $\Lambda$-coalescent from Theorem \ref{Cannings} using the analogous observations for the simple models made in Section \ref{2ptsec}. 

For any $N \in \N$, let the length of spring be $t_N = T_N = \log(N^2)$. This is an arbitrary choice that ensures that the length of spring is finite in the Cannings model, but grows to infinity as $N\rightarrow \infty$.  Let 
$$\omega_N = N^{-2},$$
and let $b_{1,N}, \ldots, b_{N,N}$ be i.i.d. Bernoulli random variables with parameter $\omega_N$. Choose the wake-up times to be of the form 
\begin{align*}
 \tau_{i,N} = (1-b_{i,N})T_N + b_{i,N}\tilde \tau_{i,N}, \quad i \in [N],
\end{align*}
where $\tilde\tau_{1,N}, \ldots, \tilde\tau_{N,N}$ are i.i.d. with a distribution that will depend on the specific form of $\Lambda$ in \eqref{eq:Lambda}, and $(b_{i,N})_{i \in[N]}$ and $(\tilde \tau_{i,N})_{i \in[N]}$ are independent, for any fixed $N \in \N$.  Here $\tilde \tau_{i,N}$ describes the wake-up time of the individual $i$ given that it woke up early. 
 Assuming $\lambda_N\equiv 1$, the number of offspring of individual $i$ at the end of spring before resampling is thus given by 
\begin{align*}
 X_{i,N} = (1-b_{i,N}) + b_{i,N}G_{i,N}
\end{align*}
where $G_{i,N}$ has a geometric distribution on $\N$ with parameter $e^{-(T_N - \tilde \tau_{i,N})}$. Abbreviate $\sigma_{i,N}:=T_N-\tilde\tau_{i,N}$. 

Like in \eqref{eq:cN_only_1_matters} in the simple cases, we immediately see 
\begin{equation}\label{eq:cNsim}
  c_N \sim N\omega_NE\left[\frac{(G_{1,N})_2}{(G_{1,N} +N - 1)_2}\right].
\end{equation}
Note that this observation holds for any distribution of $\sigma_{i,N}$.
In particular, if we denote by $c^{(1)}_N$, $c^{(2), \kappa}_N$  and $c^{(3)}_N$ the analogs of $c_N$ in the simple model with $\beta=1+r$, $\beta=1$ and $\beta=1-r$, respectively, by Lemma \ref{c1-J} we know
\begin{align}\label{eq:lemma3.1}
 N \omega_NE\left[\frac{(G_{1,N})_2}{(G_{1,N} +N - 1)_2}\;\bigg\vert\; \sigma_{1,N}=\log(N^{1+r}) \right] \sim &\ c^{(1)}_N \sim  N\omega_N,\\
 N \omega_NE\left[\frac{(G_{1,N})_2}{(G_{1,N} +N - 1)_2}\;\bigg\vert\; \sigma_{1,N}=\log(\kappa N)\right] \sim &\ c^{(2), \kappa}_N\sim N\omega_NE[Y^2_{\kappa}],\quad \kappa>0,\notag\\
 N \omega_NE\left[\frac{(G_{1,N})_2}{(G_{1,N} +N - 1)_2}\;\bigg\vert\; \sigma_{1,N}=\log(N^{1-r}) \right] \sim &\ c^{(3)}_N \sim 2\omega_NN^{1-2r}.\notag
\end{align}
Assume more specifically that $r \in (0, 1/2),$
and define
\begin{align*}
\alpha_N:=\int_{N^{-r/2}}^N \eta(\dd\kappa).
\end{align*}
Since we assume $\int_0^{\infty} (\kappa^2\wedge1)\eta(\dd \kappa)<\infty$, we can estimate
\begin{align}\label{eq:estalpha1}
 \alpha_N=\int_{N^{-r/2}}^N \eta(\dd\kappa) = \int_{N^{-r/2}}^N \frac{1}{(\kappa^2\wedge1)} (\kappa^2\wedge1)\eta(\dd\kappa) \leq N^{r}\int_0^{\infty} (\kappa^2\wedge1)\eta(\dd \kappa).
 \end{align}
In particular, $N^{-2r}\alpha_N \rightarrow 0$ as $N\rightarrow \infty$. 
 
We now distinguish different cases for the possible choices of $\Lambda$ in \eqref{eq:Lambda}. Let us first treat the case where $a_1>0$.  Using $X :=_d \mu$ to denote that we are defining a random variable $X$ to have the distribution $\mu$, for sufficiently large $N$ we can define
 \begin{align*}
  \sigma_{1,N} &:=_d \frac{2N^{-2r}\alpha_N}{a_1} \int_{N^{-r/2}}^N \delta_{(\log(\kappa N)\wedge T_N)}\frac{\eta(\dd \kappa)}{\alpha_N}\\
		& \qquad \qquad +\frac{2N^{-2r}a_2}{a_1} \delta_{(\log(N^{1+r})\wedge T_N)} + \left(1- \frac{2N^{-2r}(\alpha_N + a_2)}{a_1}\right)\delta_{(\log(N^{1-r})\wedge T_N)}\notag\\
		 & = \frac{2N^{-2r}\alpha_N}{a_1} \int_{N^{-r/2}}^{N}\delta_{\log(\kappa N)}\frac{\eta(\dd \kappa)}{\alpha_N}\\
		& \qquad \qquad +\frac{2N^{-2r}a_2}{a_1} \delta_{\log(N^{1+r})} + \left(1- \frac{2N^{-2r}(\alpha_N + a_2)}{a_1}\right)\delta_{\log(N^{1-r})}.\notag
 \end{align*}
 If $\eta$ is not the zero-measure, then $\alpha_N$ is positive and the integral in the first term is an integral with respect to the probability measure $\1_{[N^{-r/2}, N]}(\kappa)\eta(\dd \kappa)/\alpha_N$.  Otherwise the first term is simply zero. 
 
Note that this $\sigma_{1,N}$ is precisely a mixture of the simple cases discussed above. We now check the three conditions in Theorem \ref{Cannings} to obtain the desired convergence.  Let $G_{1,N}^{\kappa}$ be a random variable having a geometric distribution on $\N$ with parameter $(\kappa N)^{-1}$.
Conditioning on the possible values of $\sigma_{1,N}$ in \eqref{eq:cNsim}, and using \eqref{eq:lemma3.1} in the last step, which determines the asymptotic behavior of the last two summands, we then get
\begin{align}\label{eq:cN_1}
c_N	& \sim N  \omega_N  \bigg\{ \frac{2N^{-2r}\alpha_N}{a_1}\int_{N^{-r/2}}^{N} E\left[\frac{(G_{1,N})_2}{(G_{1,N} +N - 1)_2}\;\bigg\vert\; \sigma_{1,N}=\log(\kappa N) \right]\frac{\eta(\dd \kappa)}{\alpha_N} \notag\\
	& \qquad \qquad \qquad  + \frac{2N^{-2r} a_2}{a_1} E\left[\frac{(G_{1,N})_2}{(G_{1,N} +N - 1)_2}\;\bigg\vert\; \sigma_{1,N}=\log(N^{1+r})\right] \notag\\
	& \qquad \qquad \qquad \qquad +\left(1-2N^{-2r}\frac{\alpha_N + a_2}{a_1}\right)E\left[\frac{(G_{1,N})_2}{(G_{1,N} +N - 1)_2}\;\bigg\vert\; \sigma_{1,N}=\log( N^{1-r}) \right]\bigg\}\notag\\
	& \sim N  \omega_N\frac{2N^{-2r}}{a_1}\int_{N^{-r/2}}^{N}  E\left[\frac{(G_{1,N}^{\kappa})_2}{(G_{1,N}^{\kappa} +N - 1)_2}\right]\eta(\dd \kappa) \notag\\
	& \qquad \qquad \qquad \qquad + \frac{2N^{-2r} a_2}{a_1} c^{(1)}_N +\left(1-2N^{-2r}\frac{\alpha_N + a_2}{a_1}\right)c^{(3)}_N
  \end{align}
Due to the integral, the first summand requires a bit more care.  Equation \eqref{eq:lemma3.1} yields
  \begin{align*}
   \lim_{N\rightarrow \infty} E\left[\frac{(G_{1,N}^{\kappa})_2}{(G_{1,N}^{\kappa} +N - 1)_2}\right] = E\left[Y_{\kappa}^2\right].
  \end{align*}
By \eqref{eq:Y_is_W}, we have $Y_{\kappa} = \kappa W/(\kappa W+1)$, where $W$ has an exponential distribution with parameter~1.  Therefore, using (\ref{Ydef}) and the fact that $E[Y_{\kappa}^2] = \int_0^{\infty} P(Y_{\kappa} > \sqrt{x}) \, \dd x$, we estimate
 \begin{align}\label{eq:kappa2}
  \int_0^1e^{-\frac{1}{\kappa}\frac{\sqrt{x}}{1-\sqrt{x}}}\dd x = E\left[Y_{\kappa}^2\right] = E\left[\left(\frac{\kappa W}{\kappa W+1}\right)^2\right] \leq \kappa^2 E[W^2] = 2\kappa^2
 \end{align}
 for any $\kappa >0$. Like in the simple cases \eqref{i2b}, we obtain 
 \begin{align*}
  E\left[\frac{(G_{1,N}^{\kappa})_2}{(G_{1,N}^{\kappa} +N - 1)_2} \right] = E\left[\left(\frac{G_{1,N}^{\kappa}}{G_{1,N}^{\kappa}+N-1}\right)^2\right]-E\left[\frac{(N-1)G_{1,N}^{\kappa}}{(G_{1,N}^{\kappa}+N-1)(G_{1,N}^{\kappa}+N-1)_2}\right],             
 \end{align*}
and the second term is smaller than $1/N$. Using \eqref{eq:estalpha1}
 \begin{align}\label{eq:1/Nint}
  \int_{N^{-r/2}}^{N} E\left[\frac{(N-1)G_{1,N}^{\kappa}}{(G_{1,N}^{\kappa}+N-1)(G_{1,N}^{\kappa}+N-1)_2}\right]\eta(\dd \kappa) \leq N^{-1}N^{r}\int_0^{\infty}(\kappa^2\wedge 1)\eta(\dd \kappa) \xrightarrow{N\rightarrow \infty} 0.
 \end{align}
 Like in \eqref{i3} and below, we see
 \begin{align*}
  E\left[\left(\frac{G_{1,N}^{\kappa}}{G_{1,N}^{\kappa}+N-1}\right)^2\right] 	& = \int_{0}^1\P\left( G_{1,N}^{\kappa}>\frac{(N-1)\sqrt{x}}{1-\sqrt{x}} \right)\dd x =\int_{0}^1\left (1-\frac{1}{\kappa N}\right)^{\left\lfloor (N-1)\frac{\sqrt{x}}{1-\sqrt{x}}\right\rfloor}\dd x\\
&\leq \left(1-\frac{1}{\kappa N}\right)^{-1}\int_0^1e^{\frac{\sqrt{x}}{1-\sqrt{x}}(N-1)\log\left(1-\frac{1}{\kappa N}\right)} \dd x \\
&= \left(1-\frac{1}{\kappa N}\right)^{-1} E\left[Y_{\frac{1}{c(\kappa,N)}}^2\right],
 \end{align*}
 if we define
 \begin{align*}
  c(\kappa, N):= -(N-1)\log\left(1-\frac{1}{\kappa N}\right)>0.
 \end{align*}
 Using \eqref{eq:kappa2}, for any $N$ sufficiently large we can bound
 \begin{align}\label{eq:c2}
   E\left[\left(\frac{G_{1,N}^{\kappa}}{G_{1,N}^{\kappa}+N-1}\right)^2\right]  \leq 3c(\kappa,N)^{-2}.
 \end{align}
 Standard calculus shows that  $\kappa c(\kappa, N)$ is decreasing in $\kappa$, and we can therefore estimate 
\begin{align*}
  \inf_{N^{-r/2}\leq \kappa\leq 1} \kappa c(\kappa, N) = c(1, N) = -(N-1)\log\left(1-\frac{1}{N}\right) \xrightarrow{N\rightarrow \infty} 1.
 \end{align*}
 This yields a uniform upper bound for $c(\kappa,N)^{-1}$ because for all $N$ sufficiently large
 \begin{align}\label{eq:ckappa}
  \forall \kappa \in [N^{-r/2},1], \text{\; we have \,} \kappa c(\kappa,N)\geq \frac{1}{2} \text{ \; and therefore \; } c(\kappa,N)^{-1} \leq 2\kappa.
 \end{align}
 This allows us to estimate 
 \begin{align*}
  \1_{[N^{r/2},N]}(\kappa) E\left[\left(\frac{G_{1,N}^{\kappa}}{G_{1,N}^{\kappa}+N-1}\right)^2\right] & \leq \1_{[N^{r/2},1]}(\kappa)3c(\kappa, N)^{-2} +  \1_{(1,N]}(\kappa)\\
													& \leq \1_{[N^{r/2},1]}(\kappa)8\kappa^2 +  \1_{(1,N]}(\kappa) \\
													& \leq \1_{[N^{r/2},N]}(\kappa)12(\kappa^2\wedge 1).
 \end{align*}

Since we assume $\int_0^{\infty} (\kappa^2\wedge1)\eta(\dd \kappa)<\infty$, we have found an integrable upper bound.  Lebesgue's dominated convergence theorem and the fact that $c(\kappa, N)^{-1} \rightarrow \kappa$ as $N \rightarrow \infty$ yield
 \begin{align*}
  \lim_{N\rightarrow \infty}\int_0^{\infty}\1_{[N^{-r/2}, N]}(\kappa)E\left[\left(\frac{G_{1,N}^{\kappa}}{G_{1,N}^{\kappa}+N-1}\right)^2\right] \eta(\dd \kappa) = \int_0^{\infty}E\left[Y_{\kappa}^2\right]\eta(\dd \kappa).
 \end{align*}
 If we combine this with \eqref{eq:1/Nint}, we obtain
 \begin{align}\label{eq:Y2isthelimit}
  \lim_{N\rightarrow\infty}\int_{N^{-r/2}}^{N} E\left[\frac{(G_{1,N}^{\kappa})_2}{(G_{1,N}^{\kappa} +N - 1)_2} \right]\eta(\dd \kappa) = \int_0^{\infty}E\left[Y_{\kappa}^2\right]\eta(\dd \kappa).
 \end{align}
 It follows from (\ref{declp}) that
 \begin{equation}\label{EYkappa2}
 \int_0^{\infty} E[Y_{\kappa}^2] \: \eta(\dd \kappa) = \Lambda'([0,1]) = 1 - a_1 - a_2,
 \end{equation}
 and it was noted after (\ref{eq:estalpha1}) that $N^{-2r} \alpha_N \rightarrow 0$ as $N \rightarrow \infty$.  Therefore, plugging (\ref{eq:Y2isthelimit}) into \eqref{eq:cN_1} yields
 \begin{align}\label{eq:cNasympt}
  c_N 	& \sim N  \omega_N\frac{2N^{-2r}}{a_1}\int_0^{\infty}E\left[Y_{\kappa}^2\right]\eta(\dd \kappa) + \frac{2N^{-2r} a_2}{a_1}  N  \omega_N +\left(1-2N^{-2r}\frac{\alpha_N + a_2}{a_1}\right)2\omega_NN^{1-2r}.\notag \\
	& \sim N^{1-2r}\omega_N \left(\frac{2}{a_1}(1-a_1-a_2) + \frac{2a_2}{a_1} + 2\right) = N^{1-2r}\omega_N \frac{2}{a_1}
 \end{align}
 In particular, $c_N\rightarrow 0$ as $N\rightarrow \infty$ and the first condition of Theorem \ref{Cannings} holds.
 The second condition of Theorem \ref{Cannings} now follows directly from Lemma \ref{nu12tail}.

We are only left to verify the third condition. 
As in the proof of Theorem \ref{simplethm}, let $A := \{ \tau_{1,N} = 0\}=\{ b_{1,N} = 1\}$ be the event that the individual with label $1$ woke up early, and let $B:=\{\sum_{i=1}^N b_{i,N} =2\}$ be the event that at least two individuals woke up early.
Like in the simple case, regardless of the precise distribution of $\sigma_{1,N}$ we have 
 \begin{align*}
 \P\left(\frac{X_{1,N}}{S_N} > y\right) & =\P\left(\frac{X_{1,N}}{S_N} > y \;\big\vert\;  A^c\right)\P(A^c) + \P\left(\frac{X_{1,N}}{S_N} >y \;\big\vert\;  A\cap B\right)\P(A\cap B) \\
					& \qquad \qquad + \P\left(\frac{X_{1,N}}{S_N} > y  \;\big\vert\;  A\cap B^c\right)\P(A\cap B^c).
 \end{align*}
We have  $\P(X_{1,N}/S_N > y \;| \; A^c) \leq \P(1/N>y) = 0$ for $N$ sufficiently large, whence we may ignore this term in further considerations. Also, since $\P(A \cap B) \leq N\omega_N^2$ and we chose $\omega_N=N^{-2}$,
  \begin{align*}
   \limsup_{N\rightarrow \infty} \frac{N}{c_N}\P\left(\frac{X_{1,N}}{S_N} >y \;\big\vert\;  A\cap B\right)\P(A \cap B) \leq \lim_{N\rightarrow\infty}\frac{N}{c_N}N\omega_N^2 = \lim_{N\rightarrow\infty}\frac{a_1}{2}N^{-1+2r}=0.
  \end{align*}
As before, we condition on the different values of $\sigma_{1,N}$.
Letting $G_{1,N}^{\kappa}$, $G_{1,N}^{+}$, and $G_{1,N}^{-}$ be geometric random variables with parameters $(\kappa N)^{-1}$, $N^{-(1+r)}$ and $N^{-(1-r)}$ respectively, we obtain
 \begin{align*}
  \lim_{N\rightarrow\infty}&\frac{N}{c_N}\P\left(\frac{X_{1,N}}{S_N} > y\right)	 = \lim_{N\rightarrow\infty}\frac{N}{c_N}\P\left(\frac{G_{1,N}}{G_{1,N}+N-1}> y \right)\P(A_1\cap B_2^c)\\
			    & = \lim_{N\rightarrow\infty}\frac{N\omega_N}{c_N}\bigg\{\frac{2N^{-2r}\alpha_N}{a_1}\int_{N^{-r/2}}^{N}\P\left(\frac{G_{1,N}}{G_{1,N}+N-1}> y \;\bigg\vert\; \sigma_{1,N}=\log(\kappa N) \right)\frac{\eta(\dd \kappa)}{\alpha_N} \\
			    & \qquad \qquad  + \frac{2N^{-2r} a_2}{a_1}\P\left(\frac{G_{1,N}}{G_{1,N}+N-1}> y \;\bigg\vert\; \sigma_{1,N}=\log(N^{1+r}) \right)\\
			    & \qquad \qquad \qquad +\left(1-2N^{-2r}\frac{\alpha_N + a_2}{a_1}\right)\P\left(\frac{G_{1,N}}{G_{1,N}+N-1}> y\;\bigg\vert\; \sigma_{1,N}=\log( N^{1-r})  \right)\bigg\} \\
			    & = \lim_{N\rightarrow\infty}\int_{N^{-r/2}}^{N}\P\left(\frac{G_{1,N}^{\kappa}}{G_{1,N}^{\kappa}+N-1} > y\right)\eta(\dd \kappa) \\
			    & \qquad \qquad  +  \lim_{N\rightarrow\infty}a_2P\left(\frac{G_{1,N}^{+}}{G_{1,N}^{+}+N-1} >y\right) \\
			    & \qquad \qquad \qquad + \lim_{N\rightarrow\infty}\frac{a_1}{2}N^{2r}\P\left(\frac{G_{1,N}^{-}}{G_{1,N}^{-}+N-1} > y\right).
 \end{align*}	
  Let us consider the three limits separately.  Reasoning as in the proof of Lemma \ref{moments},
 \begin{align*}
   \lim_{N\rightarrow \infty}\P\left(\frac{G_{1,N}^{\kappa}}{G_{1,N}^{\kappa}+N-1} > y\right) = \lim_{N\rightarrow \infty}\left(1-\frac{1}{\kappa N}\right)^{\left\lfloor(N-1)\frac{y}{1-y}\right\rfloor  } = e^{-\frac{1}{\kappa}\frac{y}{1-y}}.
 \end{align*}
 Using Chebychev's inequality together with \eqref{eq:c2} and \eqref{eq:ckappa} we again obtain an integrable upper bound and therefore can use Lebesgue's dominated convergence theorem to obtain
 \begin{align}\label{eq:Pkappa}
  \lim_{N\rightarrow \infty}\int_{N^{-r/2}}^{N}\P\left(\frac{G_{1,N}^{\kappa}}{G_{1,N}^{\kappa}+N-1} > y\right)\eta(\dd \kappa) = \int_0^{\infty}e^{-\frac{1}{\kappa}\frac{y}{1-y}}\eta(\dd \kappa).
 \end{align}
 Likewise,
 \begin{align}\label{eq:Pr_1}
  \lim_{N\rightarrow \infty}\P\left(\frac{G_{1,N}^{+}}{G_{1,N}^{+}+N-1} >y\right) = \lim_{N\rightarrow \infty}\left(1-\frac{1}{ N^{1+r}}\right)^{\left\lfloor(N-1)\frac{y}{1-y}\right\rfloor  } = 1.
 \end{align}
Lastly,
 \begin{align*}
  \lim_{N\rightarrow\infty}\frac{a_1}{2}N^{2r}\P\left(\frac{G_{1,N}^{-}}{G_{1,N}^{-}+N-1} > y\right) 
				& = \lim_{N\rightarrow \infty}\frac{a_1}{2}N^{2r}\left(1-\frac{1}{ N^{1-r}}\right)^{\left\lfloor(N-1)\frac{y}{1-y}\right\rfloor  } = 0.
 \end{align*}
 Combining this, we obtain
 \begin{align}\label{eq:limP}
  \lim_{N\rightarrow\infty}&\frac{N}{c_N}\P\left(\frac{X_{1,N}}{S_N} > y\right) = \int_0^{\infty}e^{-\frac{1}{\kappa}\frac{y}{1-y}}\eta(\dd \kappa) + a_2 + 0
 \end{align}
 for every $y \in (0,1)$. Applying Lemma \ref{LLNlem} as in \eqref{eq:nu_vs_X} and using the identity
 \begin{equation}\label{expid}
e^{-\ell x/(1-x)} = \int_x^1 \frac{\ell}{(1-y)^2} e^{-\ell y/(1-y)} \: {\rm d}y,
\end{equation}
equation (\ref{eq:limP}) implies that for every $x\in (0,1)$, we have
 \begin{align}\label{eq:limnu}
   \lim_{N\rightarrow\infty}\frac{N}{c_N}\P\left(\nu_{1,N} > Nx\right) 	& = \int_0^{\infty}e^{-\frac{1}{\kappa}\frac{x}{1-x}}\eta(\dd \kappa) + a_2 + 0 \nonumber \\
									& = \int_0^{\infty}\int_x^1\frac{1}{\kappa}\frac{1}{(1-y)^2}e^{-\frac{1}{\kappa}\frac{y}{1-y}}\dd y\;\eta(\dd \kappa) + a_2 + 0 \nonumber \\
									& = \int_x^1 y^{-2}h(y)\dd y + a_2 + 0 \nonumber \\
									& = \int_x^1 y^{-2}(\Lambda' + a_2 \delta_1 + a_1 \delta_0)(\dd y).
 \end{align}
 With this we have verified the third condition of Theorem \ref{Cannings} and may thus conclude that the ancestral processes of the Cannings model we constructed do indeed converge to the $\Lambda$-coalescent.
 
 The case of $a_1=0$ just requires an adaptation of the distribution of $\sigma_{1,N}$.  We define
\begin{align}\label{eq:distsigma_2}
  \sigma_{1,N} &:=_d N^{-2r}\alpha_N \int_{N^{-r/2}}^{N}\delta_{\log(\kappa N)}\frac{\eta(\dd \kappa)}{\alpha_N} +N^{-2r}a_2 \delta_{\log(N^{1+r})} + \left(1- N^{-2r}(\alpha_N + a_2)\right)\delta_{0}.\notag
 \end{align}
 As before, we calculate $c_N$ by conditioning on the possible values of $\sigma_{1,N}$ and obtain
$$c_N \sim N  \omega_NN^{-2r}\int_{N^{-r/2}}^{N}  E\left[\frac{(G_{1,N}^{\kappa})_2}{(G_{1,N}^{\kappa} +N - 1)_2}\right]\eta(\dd \kappa) + N^{-2r} a_2 c^{(1)}_N + 0.$$
Using \eqref{eq:lemma3.1}, \eqref{eq:Y2isthelimit}, and (\ref{EYkappa2}), we get
  \begin{align*}
   c_N	 & \sim N^{1-2r}  \omega_N\left\{\int_0^{\infty}  E\left[Y^2_{\kappa}\right]\eta(\dd \kappa) + a_2 \right\} =  N^{1-2r}  \omega_N\left\{  1-a_2 + a_2 \right\}  = N^{1-2r}  \omega_N,
  \end{align*}
 which converges to 0 as $N\rightarrow \infty$ and therefore the first condition in Theorem \ref{Cannings} holds. Again, the second condition the follows directly from  Lemma \ref{nu12tail}. To obtain the third condition, as before, we condition on the different values of $\sigma_{1,N}$ and obtain
 \begin{align*}
  \lim_{N\rightarrow\infty}\frac{N}{c_N}\P\left(\frac{X_{1,N}}{S_N} > y\right)	
			    & = \lim_{N\rightarrow\infty}\int_{N^{-r/2}}^{N}\P\left(\frac{G_{1,N}^{\kappa}}{G_{1,N}^{\kappa}+N-1} > y\right)\eta(\dd \kappa) \\
			    & \qquad \qquad \qquad \qquad +  \lim_{N\rightarrow\infty}a_2P\left(\frac{G_{1,N}^{+}}{G_{1,N}^{+}+N-1} >y\right).
 \end{align*}	
Using \eqref{eq:Pkappa} and \eqref{eq:Pr_1} we obtain 
 \begin{align*}
  \lim_{N\rightarrow\infty}&\frac{N}{c_N}\P\left(\frac{X_{1,N}}{S_N} > y\right) = \int_0^{\infty}e^{-\frac{1}{\kappa}\frac{y}{1-y}}\eta(\dd \kappa) + a_2 
 \end{align*}
 for every $y \in (0,1)$ as in \eqref{eq:limP} and therefore, for every $x\in (0,1)$, 
 \begin{align*}
   \lim_{N\rightarrow\infty}\frac{N}{c_N}\P\left(\nu_{1,N} > Nx\right) 
	& = \int_x^1 y^{-2}(\Lambda' + a_2 \delta_1)(\dd y)
 \end{align*}
 as in \eqref{eq:limnu}. With this we have verified the third condition of Theorem \ref{Cannings}, which completes the proof.
 \end{proof}

\begin{proof}[Proof of Theorem \ref{charthm}]
Proposition \ref{construction} and Remark \ref{probrem} establish that all measures $\Lambda$ that can be written as in (\ref{decomp}), with the density of $\Lambda'$ being given by (\ref{hdef}), can arise as limits of the ancestral processes in the model introduced in Section \ref{modelsec}.  It remains to show that these are the only measures that can be obtained.

In view of Remark \ref{probrem}, it suffices to consider the scaling in which $\rho_N = 1/c_N$.  Note that if we denote by $\mu_N$ the distribution of $\exp(-\lambda_N(T_N - \tau_{1,N}))$, then
\begin{equation}\label{geomeq}
\P(X_{1,N} > n) = \int_0^1 (1 - p)^m \: \mu_N({\rm d}p), \hspace{.2in}\mbox{for all }m \in \N.
\end{equation}
That is, the distribution of $X_{1,N}$ is a mixture of geometric distributions.  We need to show that if (\ref{geomeq}) holds, 
then the measure $\Lambda$ that appears on the right-hand side of (\ref{X1cond}) must satisfy (\ref{decomp}) and (\ref{hdef}).

For $0 < x < 1$, we have
$$\frac{N}{c_N} \P \bigg( \frac{X_{1,N}}{a_N} > \frac{x}{1 - x} \bigg) = \frac{N}{c_N} \P \bigg(X_{1,N} > \frac{a_N x}{1 - x} \bigg) = \frac{N}{c_N} \int_0^1 (1 - p)^{\lfloor \frac{a_N x}{1-x} \rfloor} \: \mu_N({\rm d}p).$$
We first show that we get minimal contribution to the integral when $p \geq N^{-3/4}$.  Using (\ref{cNlower}),
\begin{align*}
\frac{N}{c_N} \int_{N^{-3/4}}^1 (1 - p)^{\lfloor \frac{a_N x}{1-x} \rfloor} \: \mu_N({\rm d}p) &\leq \frac{N^2(N+1)}{2 \P(S_N > N)} \int_{N^{-3/4}}^1 (1 - p)^{\lfloor \frac{a_N x}{1-x} \rfloor} \: \mu_N({\rm d}p) \\
&\leq \frac{N^2(N+1)}{2 \P(X_{1,N} > 1)} \int_{N^{-3/4}}^1 (1 - p)^{\lfloor \frac{a_N x}{1-x} \rfloor} \: \mu_N({\rm d}p) \\
&= \frac{N^2(N+1)}{2} \frac{\int_{N^{-3/4}}^1 (1 - p)^{\lfloor \frac{a_N x}{1-x} \rfloor} \: \mu_N({\rm d}p)}{\int_{N^{-3/4}}^1 (1 - p) \: \mu_N({\rm d}p)}.
\end{align*}
Because $a_N \geq N$, we have, for all $p \geq N^{-3/4}$,
$$\frac{(1-p)^{\lfloor \frac{a_N x}{1-x} \rfloor}}{1-p} \leq \bigg(1 - \frac{1}{N^{3/4}} \bigg)^{\frac{Nx}{1-x} - 2},$$
and therefore $$\limsup_{N \rightarrow \infty} \frac{N}{c_N} \int_{N^{-3/4}}^1 (1 - p)^{\lfloor \frac{a_N x}{1-x} \rfloor} \: \mu_N({\rm d}p) \leq \limsup_{N \rightarrow \infty} \frac{N^2(N + 1)}{2} \bigg(1 - \frac{1}{N^{3/4}} \bigg)^{\frac{Nx}{1-x} - 2} = 0.$$  It follows that for all $x \in (0,1)$ such that $\Lambda(\{x\}) = 0$, we have
\begin{equation}\label{Lambdalim}
\lim_{N \rightarrow \infty} \frac{N}{c_N} \int_0^{N^{-3/4}} (1 - p)^{\lfloor \frac{a_N x}{1-x} \rfloor} \: \mu_N({\rm d}p) = \int_x^1 y^{-2} \: \Lambda({\rm d}y).
\end{equation}

We claim that for $p < N^{-3/4}$, we can make the approximation $1 - p \approx e^{-p}$, and we therefore assume for now that
\begin{equation}\label{expapprox}
\lim_{N \rightarrow \infty} \frac{N}{c_N} \int_0^{N^{-3/4}} \Big( (1 - p)^{\lfloor \frac{a_N x}{1-x} \rfloor} - e^{-\frac{a_N px}{1-x}} \Big) \: \mu_N({\rm d}p) = 0.
\end{equation}
Then, (\ref{Lambdalim}) implies that for all $x \in (0,1)$ such that $\Lambda(\{x\}) = 0$, we have
$$\lim_{N \rightarrow \infty} \frac{N}{c_N} \int_0^{N^{-3/4}} e^{-\frac{a_N px}{1-x}} \: \mu_N({\rm d}p) = \int_x^1 y^{-2} \: \Lambda({\rm d}y).$$  Define a new measure $\chi_N$ on $(0, \infty)$ to be the push-forward of the restriction of $\mu_N$ to $(0, N^{-3/4})$ by the map $p \mapsto a_N p$, multiplied by $N/c_N$.  Writing $z = x/(1-x)$, we then have
\begin{equation}\label{LapLambda}
\lim_{N \rightarrow \infty} \int_0^{\infty} e^{-\ell z} \: \chi_N(d \ell) = \int_x^1 y^{-2} \: \Lambda({\rm d}y)
\end{equation}
for all $x \in (0,1)$ such that $\Lambda(\{x\}) = 0$.  We claim that this convergence must hold for all $x \in (0,1)$.  To see this, we assume, seeking a contradiction, that $\Lambda(\{x\}) = b > 0$ for some $x > 0$.  Choose $u$ and $v$ such that $0 < u < v < z$.  Choose $C_1$ and $C_2$ such that $e^{-\ell v} \ell \leq C_1 e^{-\ell u}$ for all $\ell \geq 0$ and $\int_0^{\infty} e^{-\ell u} \chi_N({\rm d} \ell) \leq C_2$ for sufficiently large $N$.  Choose $0 < \delta < \min\{z - v, b/(4x^2C_1C_2)\}$.  Then, (\ref{LapLambda}) implies that for sufficiently large $N$, we have
\begin{align*}
\frac{b}{2x^2} &< \int_0^{\infty} \big( e^{-\ell(z - \delta)} - e^{-\ell(z + \delta)} \big) \: \chi_N({\rm d}\ell) \\
&\leq 2 \delta \int_0^{\infty} e^{-\ell(z - \delta)} \ell \: \chi_N({\rm d}\ell) \\
&\leq 2 \delta C_1 \int_0^{\infty} e^{-\ell u} \: \chi_N({\rm d}\ell) \\
&\leq 2 \delta C_1 C_2 \\
&\leq \frac{b}{2x^2},
\end{align*}
which is a contradiction.  Therefore, $\Lambda(\{x\}) = 0$ for all $x > 0$, and thus by (\ref{LapLambda}) the Laplace transforms of the measures $\chi_N$ converge pointwise to a limit on $(0, \infty)$.  By Theorem 8.5 of \cite{bw}, it follows that the measures $\chi_N$ converge vaguely to a limit measure $\chi$, and the pointwise limit of the Laplace transforms of $\chi_N$ is the Laplace transform of $\chi$.  That is, for all $x \in (0, 1)$, we have
$$\int_0^{\infty} e^{-\ell x/(1-x)} \: \chi({\rm d}\ell) = \int_x^1 y^{-2} \: \Lambda({\rm d}y).$$  We now use (\ref{expid})
and change the order of integration to get
$$\int_x^1 \bigg( \int_0^{\infty} \ell \bigg( \frac{y}{1-y} \bigg)^2 e^{-\ell y/(1-y)} \: \chi({\rm d} \ell) \bigg) \: y^{-2} \: {\rm d}y = \int_x^1 y^{-2} \: \Lambda({\rm d}y).$$  Now letting $\eta$ be the push-forward of $\chi$ by the map $x \mapsto 1/x$, we see that the restriction of $\Lambda$ to $(0,1)$ must have density $h$, as given in (\ref{hdef}).  To obtain the integrability condition, we note that because we are assuming $\rho_N = 1/c_N$, the measure
$\Lambda$ must be a probability measure, and therefore $$1 \geq \int_0^1 h(y) \: {\rm d}y = \int_0^{\infty} \int_0^1 y^2 \cdot \frac{1}{\kappa (1-y)^2} e^{-\frac{y}{\kappa(1-y)}} \: {\rm d}y \: \eta({\rm d}\kappa).$$  Letting $W$ have an exponential distribution with mean $1$, the inner integral is $$E \bigg[ \bigg( \frac{\kappa W}{\kappa W + 1} \bigg)^2 \bigg],$$ which is easily seen to be bounded between $C_3(1 \wedge \kappa^2)$ and $C_4(1 \wedge \kappa^2)$ for some positive constants $C_3$ and $C_4$ for all $\kappa > 0$.  This implies that $\int_0^{\infty} (1 \wedge \kappa^2) \: \eta({\rm d} \kappa) < \infty$.

It remains only to establish (\ref{expapprox}).  For $0 < p < 1$, the Taylor expansion $\log(1-p) = -\sum_{n=1}^{\infty} p^n/n$ yields $-p/(1-p) \leq \log(1-p) \leq -p$, and therefore $e^{-p/(1-p)} \leq 1-p \leq e^{-p}$.  It follows that
\begin{equation}\label{squeeze}
e^{-\frac{p^2}{1-p} \frac{a_N x}{1-x}} e^{-\frac{a_N p x}{1-x}} = e^{-\big(\frac{p}{1-p} \frac{a_N x}{1-x} \big)} \leq (1 - p)^{\lfloor \frac{a_N x}{1-x} \rfloor} \leq e^{-p\big( \frac{a_N x}{1-x} - 1\big)} = e^{-\frac{a_N p x}{1-x}} e^p.
\end{equation}
For $p \leq N^{-3/4}$, the upper bound in (\ref{squeeze}) gives $$e^{-\frac{a_N px}{1-x}}  \geq (1 - p)^{\lfloor \frac{a_N x}{1-x} \rfloor} e^{-N^{-3/4}},$$ which, in combination with the finiteness of the right-hand side of (\ref{Lambdalim}), implies that 
\begin{equation}\label{limsup}
\limsup_{N \rightarrow \infty} \frac{N}{c_N} \int_0^{N^{-3/4}} \Big( (1 - p)^{\lfloor \frac{a_N x}{1-x} \rfloor} - e^{-\frac{a_N px}{1-x}} \Big) \: \mu_N({\rm d}p) \leq 0.
\end{equation}
For the bound in the other direction, note that if $p \leq N^{-3/4}$ and $a_N \leq N^{5/4}$, then the lower bound in (\ref{squeeze}) gives $$e^{-\frac{a_N px}{1-x}} \leq (1 - p)^{\lfloor \frac{a_N x}{1-x} \rfloor} e^{\frac{2x}{(1-x)} N^{-1/4}},$$ which, in combination with the finiteness of the right-hand side of (\ref{Lambdalim}), gives that for $a_N \leq N^{5/4}$,
\begin{equation}\label{liminf}
\liminf_{N \rightarrow \infty} \frac{N}{c_N} \int_0^{N^{-3/4}} \Big( (1 - p)^{\lfloor \frac{a_N x}{1-x} \rfloor} - e^{-\frac{a_N px}{1-x}} \Big) \: \mu_N({\rm d}p) \geq 0.
\end{equation}
Now suppose instead $a_N > N^{5/4}$.  If $p \leq N^{1/4}/a_N$ and $p \leq N^{-3/4}$, then $p^2 a_N \leq N^{-1/2}$.  Therefore, the lower bound in (\ref{squeeze}) gives $$e^{-\frac{a_N px}{1-x}} \leq (1 - p)^{\lfloor \frac{a_N x}{1-x} \rfloor} e^{\frac{2x}{(1-x)} N^{-1/2}},$$ and therefore
\begin{equation}\label{bigaN1}
\liminf_{N \rightarrow \infty} \frac{N}{c_N} \int_0^{N^{1/4}/a_N} \Big( (1 - p)^{\lfloor \frac{a_N x}{1-x} \rfloor} - e^{-\frac{a_N px}{1-x}} \Big) \: \mu_N({\rm d}p) \geq 0.
\end{equation}
Finally, to handle the case when $N^{1/4}/a_N < p \leq N^{-3/4}$, note that the assumption $a_N > N^{5/4}$ implies, by part~1 of Lemma \ref{iidXS}, that $\P(S_N \geq \frac{1}{2} N^{5/4}) \geq 1/2$ for sufficiently large $N$, and therefore there is a positive constant $C_5$ such that $c_N \geq C_5/N$.  Therefore,
\begin{equation}\label{bigaN2}
\liminf_{N \rightarrow \infty} \frac{N}{c_N} \int_{N^{1/4}/a_N}^{N^{-3/4}} \Big( (1 - p)^{\lfloor \frac{a_N x}{1-x} \rfloor} - e^{-\frac{a_N px}{1-x}} \Big) \: \mu_N({\rm d}p) \geq \liminf_{N \rightarrow \infty} - \frac{N^2}{C_5} e^{-\frac{x}{1-x} N^{1/4}} = 0.
\end{equation}
Equations (\ref{bigaN1}) and (\ref{bigaN2}) imply that (\ref{liminf}) holds also when $a_N > N^{5/4}$ which, along with (\ref{limsup}), implies that (\ref{expapprox}) holds.
\end{proof}

\bigskip
\noindent {\bf {\Large Acknowledgments}}

\bigskip
\noindent The authors thank Jochen Blath for bringing to their attention the reference \cite{wv19}.  They also thank Anton Wakolbinger for a fruitful discussion over Zoom during the Bernoulli-IMS One World Symposium.
FC was supported by the Deutsche Forschungsgemeinschaft (CRC 1283 ``Taming Uncertainty'',
Project C1).  AGC was supported in part by CONACYT CIENCIA BASICA A1-S-14615.  JS was supported in part by NSF Grant DMS-1707953.

\end{document}